\documentclass[10pt]{article}

\usepackage{amsmath,amsthm,amscd,latexsym}
\usepackage{amsfonts,color,graphicx}
\usepackage[psamsfonts]{amssymb}
\usepackage[noadjust]{cite}
\usepackage{accents}

\usepackage{listings}
\usepackage{multicol}
\usepackage{float}
\usepackage{hyperref}
\hypersetup{hidelinks}
\usepackage{tikz-cd}

\usepackage{booktabs}
\usepackage{array}

\usepackage[a4paper,margin=2cm]{geometry}

\newtheoremstyle{mystyle}               
{}                
{}                
{}        
{}                
{\bfseries \itshape}       
{.}      
{ }      
{}       

\newtheorem{theorem}{Theorem}[section]
\newtheorem{proposition}[theorem]{Proposition}
\newtheorem{lemma}[theorem]{Lemma} 
\newtheorem{corollary}[theorem]{Corollary}
\theoremstyle{definition}
\newtheorem{definition}[theorem]{Definition}
\newtheorem{example}[theorem]{Example}	
	
\theoremstyle{mystyle}
\newtheorem{remark}[theorem]{Remark}
\newtheorem*{remark*}{Remark}
\numberwithin{equation}{section}

\title{Explicit Green's Functions and Adjoint Problems for Differential Equations with Linear Functional Perturbations}
\date{ }
\author{Alberto Cabada$^{1,2}$, Paula Cambeses-Franco$^{1,2}$ and Luc{\' i}a L\'opez-Somoza$^{1,2}$\\
	$^1$ CITMAga, 15782, Santiago de Compostela, Galicia, Spain\\
	$^2$ Departamento de Estatística, Análise Matemática e Optimización\\
	Facultade de Matem\'aticas, Universidade de Santiago de Com\-pos\-te\-la, Spain.\\
	alberto.cabada@usc.es; paula.cambeses.franco@usc.es; lucia.lopez.somoza@usc.es \\
	ORCID: 0000-0003-1488-935X; 0009-0005-8030-9108; 0000-0002-9167-0709}
\begin{document}
	\maketitle
	\begin{abstract}
We study linear differential equations of arbitrary order coupled to two-point boundary conditions, perturbed by a finite sum of terms $\gamma_k(t) \mathcal{C}_k(u)$, where each $\mathcal{C}_k$ is a linear continuous functional. This framework includes equations with piecewise constant arguments, multipoint dependence and Fredholm integro-differential equations with degenerate kernels. We obtain the expression of the Green's function of the perturbed problem explicitly in terms of that of the unperturbed one, and we compute the corresponding adjoint problem. We show explicitly that the adjoint of an equation with piecewise constant arguments is an impulsive differential equation whose impulses depend on integrals of the solution or, when the evaluation points lie at the endpoints of the interval, a problem with nonlocal boundary conditions. Since the Green's function of the adjoint problem is the transpose of the original one, explicit expressions and constant-sign properties can be transferred among these three classes of problems. We also derive comparison principles that give sufficient conditions for the Green's functions to have constant sign. Finally, we obtain the exact constant-sign region of the Green's function for a class of higher-order impulsive problems with integral impulses (in particular, for a clamped beam with an impulse at its midpoint) and for a first-order periodic integro-differential problem.
	\end{abstract}



\section{Introduction}
The study of functional differential equations can be traced back to the pioneering work of Volterra \cite{volterra1928theorie}, who introduced integral terms and dependence on past states into differential equations. These ideas opened a broad field of research that was subsequently developed and systematized in the monographs of Hale \cite{hale1977theory} and Kolmanovskii and Myshkis \cite{kolmanovskii1992applied}, among many others, and continues to attract considerable attention (see, for instance, \cite{MR5119146}).

In this paper, we consider linear differential equations in which the functional dependence is represented by finitely many terms of the form $\gamma_k(t) \mathcal{C}_k(u)$ where each $\mathcal{C}_k$ is a continuous linear functional. This formulation encompasses several important classes of problems. These include differential equations with piecewise constant arguments, usually referred to as EPCA \cite{cooke1984retarded, cabada2004green, nieto2005green, buedo2024boundary}, Fredholm integro-differential equations with separable kernels \cite{polyanin2008handbook, yuldashev2018nonlocal}, and differential equations subject to nonlocal boundary conditions \cite{webb2008positive, cabada2021green}. As we shall show, it is also closely connected with impulsive differential equations \cite{milman1960stability, bainov1989systems, cabada2000green, li2007robust, domoshnitsky2014sign}.

Our analysis is based on Green's functions. Besides providing an explicit representation of the solution of the corresponding linear problem, Green's functions play a central role in the study of nonlinear equations. In particular, regions in which the Green's function has constant sign provide the natural framework for applying lower and upper solution techniques, monotone iterative methods, and fixed-point theorems in cones \cite{cabada2014greens, cabada2026firstorder}.

Two classical ideas underlie our approach. The first is that a perturbation involving finitely many linear functionals should produce a finite-dimensional correction of the Green's function of the unperturbed problem. This principle is implicit in the treatment of nonlocal boundary conditions by Webb and Infante \cite{WebbInfante2006,WebbInfante2009} and by Webb, Infante and Franco \cite{webb2008positive}, as well as in our previous study of an EPCA with reflection \cite{cabada2025reflection} and in the recent work \cite{zhumatov2026family} on loaded equations with piecewise constant arguments. In \cite{cabada2021green}, formulas of this type were obtained when only the boundary conditions change, while the differential equation remains the same. Here, by contrast, general linear functionals act directly on the differential equation. We show that the the resulting Green's function can still be obtained from that of the unperturbed problem through an explicit finite-dimensional correction.

The second idea comes from the theory of adjoint differential problems. It is well known that point evaluations or integral terms involving the solution may be transformed, at the level of the adjoint problem, into boundary terms or jump conditions at interior points \cite{Krall1965, Zettl1968, Krall1971}. Adjoints of differential operators with interior-point or Stieltjes boundary conditions, and of differential-boundary operators, in which values of the solution appear in the equation, were studied by Krall \cite{Krall1973,Krall1977}, Brown and Krall \cite{BrownKrall1974,BrownKrall1977}, Brown \cite{Brown1975,Brown1976}, Bryan \cite{Bryan1977} and Tvrdý and Vejvoda \cite{TvrdyVejvoda1972,TvrdyVejvoda1973} (see also the monograph \cite{SchwabikTvrdyVejvoda1979}, and \cite{Henry1971,Hughes1975} for adjoints of linear functional differential equations). We refer to these works for the general theory.

The novelty of the present paper lies not in the abstract construction of the adjoint, but in deriving its explicit form and Green's function in a setting that makes it possible to transfer constant-sign properties between different classes of problems. More precisely, for the scalar $n$-th order problem considered here, we show that point evaluations in the original equation lead to jump conditions involving integral functionals in the adjoint problem when the evaluation points lie in the interior of the interval, and to modified, generally nonlocal, boundary conditions when they coincide with its endpoints. This establishes an explicit correspondence between EPCA, impulsive equations with integral impulses, and differential equations with nonlocal boundary conditions, allowing Green's functions and their constant-sign regions to be transferred from one class to another.



More precisely, we consider the problem
\begin{equation}
	\left\{
	\begin{aligned}
		L_n u(t)+\sum_{k=1}^{l}\gamma_{k}(t)\mathcal{C}_{k}(u) &= \sigma(t), \quad && t \in I := [a, b], \\
		V_i(u) &= 0, \quad && i = 1, \ldots, n,
	\end{aligned}
	\right.
	\label{proprin}
\end{equation}
with the two-point boundary conditions
\begin{equation}
	V_i(u) = \sum_{j=0}^{n-1}{\left(\alpha^{i}_{j}u^{(j)}(a)+\beta^{i}_{j}u^{(j)}(b) \right)}, \quad i = 1, \ldots, n,
	\label{defv}
\end{equation}
where
\begin{equation}
	L_{n}u(t) \equiv u^{(n)}(t)+a_{1}(t)u^{(n-1)}(t)+\cdots+a_{n-1}(t)u'(t)+a_{n}(t)u(t), \quad t \in I,
	\label{defl}
\end{equation}
$\alpha^{i}_{j},\beta^{i}_{j}\in\mathbb{R}$, $\sigma\in L^1(I)$, $a_k\in L^1(I)$ for $k=1,\ldots,n$, and, for $k=1,\ldots,l$, $\gamma_k\in L^1(I)$ and $\mathcal{C}_{k}:W^{n,1}(I) \rightarrow \mathbb{R}$ is a linear continuous functional.

Here $L^p(I)$, $1\le p\le\infty$, denotes the usual Lebesgue space on $I$ with norm $\|\cdot\|_p$, and $W^{n,1}(I)$ is the space of functions $u\in C^{n-1}(I)$ such that $u^{(n-1)}$ is absolutely continuous on $I$, with the norm $\|u\|_{W^{n,1}}=\sum_{j=0}^{n}\|u^{(j)}\|_1$. 

Since each $\mathcal{C}_{k}$ is continuous on $W^{n,1}(I)$, it may depend on $u$ and its derivatives up to order $n-1$, and on $u^{(n)}$ only through integrals.

Our main contributions are the following:

\begin{enumerate}
\item We deduce an explicit expression for the Green's function of the problem for arbitrary order $n$ (Theorem \ref{teoprincipal}). We also establish comparison results for problems associated with different coefficient functions $\gamma_k$ (Proposition \ref{propsigno}).

\item When the functionals are point evaluations of the solution or of its derivatives, we explicitly determine the corresponding adjoint problem (Theorem \ref{adjpiece2}). We show that the adjoint of an EPCA is an impulsive differential equation whose jumps depend on integrals of the solution. When the evaluation points coincide with the endpoints of the interval, the adjoint becomes a differential problem with nonlocal boundary conditions.

\item We obtain an explicit Green's function for the resulting impulsive problems (Theorem \ref{obtergcomplicada}).

\item We show that explicit formulas and constant-sign regions can be transferred between the original and adjoint problems through the identity $H^*(t,s)=H(s,t)$ (with $H$ the related Green's function).
As an illustration, a constant-sign result for a first-order impulsive equation with reflection is obtained directly from known results for the corresponding EPCA (Corollary \ref{cor62}). 

\item Using a recent result by Labovskiy \cite{Lab}, we give sufficient conditions, which are also necessary when $2\le p\le n-2$, for the Green's function of the problem $u^{(n)}+M\sum_{k=1}^{l}\gamma_k(t)\,u(t_k)=\sigma(t)$, coupled to $n-p$ homogeneous conditions at $t=a$ and $p$ at $t=b$, and of its adjoint impulsive problem with integral impulses, to have constant sign (Theorem \ref{thm:imp}). The bounds are given by the spectral radii of three $l\times l$ matrices. As an application, we obtain the exact constant-sign region for a clamped beam with an impulse at its midpoint (Corollary \ref{corlab}).

\item We determine the exact constant-sign region of the Green's function associated with a first-order periodic integro-differential problem (Section \ref{sec63}).
\end{enumerate}


The paper is organized as follows. In Section \ref{explicit}, we deduce an explicit expression for the Green's function, while Section \ref{comparison} presents the corresponding comparison results. Section \ref{adjoint} deals with the adjoint problem and the constant-sign properties of the Green's function for impulsive problem. Finally, Section \ref{illustrative} provides several illustrative examples.

\section{Explicit Green's Function of the Perturbed Problem}
\label{explicit}
This section focuses on obtaining the explicit form of the solution to the general Problem \eqref{proprin}. To that end, we consider the unperturbed problem
\begin{equation}
	\left\{
	\begin{aligned}
		L_n u(t) &= \sigma(t), \quad && t \in I, \\
		V_i(u) &= 0, \quad && i = 1, \ldots, n,
	\end{aligned}
	\right.
	\label{partir}
\end{equation}
where $V_{i}$ and $L_{n}$ are defined in \eqref{defv} and \eqref{defl}, respectively. We will work under the assumption that Problem \eqref{partir} has a unique Green's function denoted by $G$, or, equivalently, that the homogeneous problem associated with \eqref{partir} has only the trivial solution.
 Taking this into account, and by the definition of the Green's function, the solution to Problem \eqref{proprin} satisfies:
\begin{equation*}
	u(t) = \int_{a}^{b} G(t,s) \left(\sigma(s) - \sum_{k=1}^{l} \gamma_{k}(s) \mathcal{C}_{k}(u) \right) \, ds.
\end{equation*}
In the next result, assuming appropriate conditions on the spectrum of Problem \eqref{partir}, we establish the existence and uniqueness of the solution to Problem \eqref{proprin}. In addition, we obtain the corresponding expression for its related Green's function. Our approach uses results similar to those in \cite{cabada2025reflection}.

Here $\mathcal{C}_j:W^{n,1}(I) \rightarrow \mathbb{R}$ is linear and continuous. Since the Green's operator $\sigma \rightarrow \int_a^b G(\cdot,s) \sigma(s) \mathrm{d}s$ is bounded from $L^1(I)$ into $W^{n,1}(I)$, there exists a unique $\kappa_j \in L^\infty(I)$ such that
$\mathcal{C}_j \left(\int_a^b G(\cdot,s) \sigma(s) \mathrm{d}s \right)=\int_a^b \kappa_j(s) \sigma(s) \mathrm{d}s$ for all $\sigma \in L^1(I)$, and we write $\mathcal{C}_j(G(\cdot,s)):=\kappa_j(s)$. If $\mathcal{C}_j$ only involves $u, \ldots, u^{(n-1)}$, this coincides for a.e. $s \in I$ with the value obtained by applying $\mathcal{C}_j$ directly to $G(\cdot,s)$. With this definition, the coefficients in \eqref{defa} are well defined and the interchange of $\mathcal{C}_j$ and the integral in the proof of Theorem \ref{teoprincipal} holds by construction.
\begin{theorem}
\label{teoprincipal}
Assume that Problem \eqref{partir} has $G$ as its unique Green's function and let $\mathcal{C}_{k}: W^{n,1}(I) \rightarrow \mathbb{R}$ be linear and continuous and $\gamma_{k} \in L^{1}(I)$, $k=1, \ldots, l$. Let $A=(a_{ij})_{l\times l}$ be given by
\begin{equation}
\label{defa}
		a_{ij}=\int_a^b \mathcal{C}_i(G(\cdot,s))\,\gamma_j(s)\,\mathrm{d}s,\qquad i,j\in\{1,\ldots,l\}.
	\end{equation}
Then, Problem \eqref{proprin} has a unique solution for every $\sigma \in L^1(I)$ if and only if 
\begin{equation}
	\det{(A+I_{l})} \neq 0,
	\label{condiciondet}
\end{equation}
where $I_l$ is the identity matrix of order $l$. In this case, the solution $u\in W^{n,1}(I)\hookrightarrow C^{n-1}(I)$ is given by
\begin{equation*}
	u(t)=\int_{a}^{b}{H(t,s) \sigma(s)\mathrm{d}s},
	\end{equation*}
where 
\begin{equation}
	H(t,s)=G(t,s)-\sum_{i=1}^{l}\int_{a}^{b}G(t,r) \gamma_{i}(r)\mathrm{d}r \sum_{j=1}^{l} \tilde{a}_{i,j}\mathcal{C}_{j}(G(\cdot,s)),
	\label{ecprincipal}
\end{equation}
with $\tilde{A}=(\tilde{a}_{ij})_{l \times l}$ denoting the inverse of matrix $A+I_l$.
\end{theorem}
\begin{proof}
Assume first that \eqref{condiciondet} holds.
Since $G$ is the unique Green's function of Problem \eqref{partir}, the solution of Problem \eqref{proprin} satisfies the following expression:
\begin{equation}
	\label{primdemo}
\begin{aligned}
	u(t)=\int_{a}^{b}{G(t,s) \left(\sigma(s)- \sum_{i=1}^{l}{\gamma_{i}(s) \mathcal{C}_{i}(u)} \right) \mathrm{d}s}
	=\int_{a}^{b}{G(t,s) \sigma(s) \mathrm{d}s}-\sum_{i=1}^{l} \mathcal{C}_{i}(u)\int_{a}^{b}{G(t,s) \gamma_{i}(s) }\mathrm{d}s.
\end{aligned}
\end{equation}
Next, we apply the linear continuous operators $\mathcal{C}_{j}$ on both sides of equation \eqref{primdemo}, which leads to
\begin{equation*}
\mathcal{C}_{j}(u)=\int_{a}^{b}{\mathcal{C}_{j}(G(\cdot,s))\sigma(s) \mathrm{d}s}-\sum_{i=1}^{l}{\mathcal{C}_{i}(u)\int_{a}^{b}{\mathcal{C}_{j}(G(\cdot,s)) \gamma_{i}(s)} \mathrm{d}s}, \quad j=1, \ldots, l.
\end{equation*}
Hence, we conclude that
\begin{equation}
	\label{sistema}
	\left( A+I_{l} \right)
	\begin{pmatrix}
		\mathcal{C}_1(u) \\
		\mathcal{C}_2(u) \\
		\vdots \\
		\mathcal{C}_l(u)
	\end{pmatrix}
	=
	\begin{pmatrix}
		\int_{a}^{b}{\mathcal{C}_{1}(G(\cdot,s)) \sigma(s) \mathrm{d}s} \\
		\int_{a}^{b}{\mathcal{C}_{2}(G(\cdot,s)) \sigma(s) \mathrm{d}s}  \\
		\vdots \\
		\int_{a}^{b}{\mathcal{C}_{l}(G(\cdot,s)) \sigma(s) \mathrm{d}s} 
	\end{pmatrix},
\end{equation}
where $I_{l}$ is the identity matrix of order $l$ and $A$ is given by equation \eqref{defa}.

From the previous equality, we obtain that
\begin{equation*}
	\mathcal{C}_{i}(u)=\sum_{j=1}^{l}{\tilde{a}_{ij}\int_{a}^{b}{\mathcal{C}_{j}(G(\cdot,s)) \sigma(s)} \mathrm{d}s}, \quad i=1, \ldots, l.
\end{equation*}
Substituting this expression into equation \eqref{primdemo}, we deduce that
\begin{equation*}
	u(t)=\int_{a}^{b}{G(t,s) \sigma(s) \mathrm{d}s}-\sum_{i=1}^{l} \sum_{j=1}^{l}{\tilde{a}_{ij} \int_{a}^{b}{\mathcal{C}_{j}(G(\cdot,r)) \sigma(r) \mathrm{d}r} \int_{a}^{b}{G(t,s) \gamma_{i}(s)\mathrm{d}s}}.
\end{equation*}
Subsequently, taking into account that
\begin{equation*}
	\int_{a}^{b} \int_{a}^{b} \mathcal{C}_{j}(G(\cdot,s)) \, \sigma(s) \, G(t,r) \, \gamma_{i}(r) \, \mathrm{d}r \, \mathrm{d}s 
	= 
	\int_{a}^{b} \mathcal{C}_{j}(G(\cdot,r)) \left( \int_{a}^{b} G(t,s) \, \gamma_{i}(s) \, \mathrm{d}s \right) \sigma(r) \, \mathrm{d}r,
\end{equation*}
we arrive at the following expression
\begin{equation*}
	u(t) = \int_{a}^{b} \left( G(t,s) - \sum_{i=1}^{l} \sum_{j=1}^{l} \tilde{a}_{ij} \, \mathcal{C}_{j}(G(\cdot,s)) \int_{a}^{b} G(t,r) \, \gamma_{i}(r) \, \mathrm{d}r \right) \sigma(s) \, \mathrm{d}s.
\end{equation*}
Finally, we conclude that
\begin{equation*}
	u(t)=\int_{a}^{b}{H(t,s) \sigma(s)} \mathrm{d}s,
\end{equation*}
where $H$ given by \eqref{ecprincipal} is the unique Green's function of Problem \eqref{proprin}.

We now show that this solution is unique. Suppose that there exists another solution $v \neq u$ to equation \eqref{proprin}, and set $w=u-v$. Then, $w$ solves the homogeneous problem associated with \eqref{proprin}, that is, $L_nw+\sum_{k=1}^{l}\gamma_k\,\mathcal{C}_k(w)=0$ on $I$ and $V_i(w)=0$, $i=1,\ldots,n$. By \eqref{sistema} with $\sigma=0$, we have $(A+I_l)\,\big(\mathcal{C}_1(w),\ldots,\mathcal{C}_l(w)\big)^T=0$, from which we deduce that $\mathcal{C}_j(w)=0$ for $j=1,\ldots,l$.
 Finally, taking into account \eqref{primdemo} and the fact that $\sigma=0$, we conclude that $w=0$, and hence $u=v$. 

	Conversely, assume that $\det(A+I_l)=0$, and let $c=(c_1,\ldots,c_l)^T\neq 0$ be such that $(A+I_l)c=0$. The function
	\begin{equation*}
	w(t)=-\sum_{i=1}^{l}c_i\int_a^b G(t,s)\,\gamma_i(s)\,\mathrm{d}s
	\end{equation*}
	satisfies $L_nw=-\sum_{i=1}^{l}c_i\gamma_i$ and $V_i(w)=0$, $i=1,\ldots,n$. Moreover, $\mathcal{C}_j(w)=-\sum_{i=1}^{l}a_{ji}c_i=c_j$ for $j=1,\ldots,l$. Hence $w$ is a nontrivial solution of the homogeneous problem associated with \eqref{proprin}, and Problem \eqref{proprin} cannot have a unique solution.
\end{proof}
In the particular case where $l=1$ we can deduce the following corollary.
\begin{corollary}
\label{casosimple}
Assume that Problem \eqref{partir} has $G$ as its unique Green's function and let $\sigma \in L^{1}(I)$, $\gamma \in L^{1}(I)$ and let $\mathcal{C}: W^{n,1}(I) \rightarrow \mathbb{R}$ be linear and continuous, such that
\[1+\int_{a}^{b}{\mathcal{C}(G(\cdot,s)) \gamma(s)} \mathrm{d}s \neq 0.\]
Then, Problem 
	\begin{equation}
	\left\{
	\begin{aligned}
		L_nu(t)+\gamma(t)\mathcal{C}(u) &= \sigma(t), \quad && t \in I := [a, b], \\
		V_i(u) &= 0, \quad && i = 1, \ldots, n,
	\end{aligned}
	\right.
	\label{propart}
\end{equation}	where $V_{i}$ and $L_{n}$ are defined in \eqref{defv} and \eqref{defl}, respectively, has a unique solution $u$, given by
\begin{equation*}
	u(t)=\int_{a}^{b}{H(t,s) \sigma(s) \mathrm{d}s},
\end{equation*}
with
\begin{equation}
	H(t,s)=G(t,s)-\mathcal{C}(G(\cdot,s)) \frac{\int_{a}^{b}{G(t,r)\gamma(r)}\mathrm{d}r}{1+\int_{a}^{b}{\mathcal{C}(G(\cdot,r)) \gamma(r) \mathrm{d}r}}.
	\label{hsimple}
\end{equation}
\end{corollary}
We state the following proposition (deduced from \cite[Definition 1.4.1]{cabada2014greens}). 
\begin{proposition}
	\label{caracterizarh}
	Assume that $a_1, \ldots, a_n$, $\gamma_1, \ldots, \gamma_l$ and $\mathcal{C}_1(G(\cdot,s)), \ldots, \mathcal{C}_l(G(\cdot,s))$ are piecewise continuous on $I$, and let $D$ be the (finite) set of their discontinuity points. Then the Green's function $H$ of Problem \eqref{proprin} satisfies the following properties:
	\begin{itemize}
		\item For each $s \in I$, $H(\cdot,s)$ is well-defined and (if $n \geq 2$) of class $C^{n-2}$ for all $t \in I$, of class $C^{n-1}$ for all $t \in I$, $t \neq s$, and of class $C^n$ for all $t \in I$, $t \neq s$, $t \notin D$.
		\item For each $s \in \mathring{I}$, the function $H(\cdot,s)$ is the solution of the following differential equation:
		\begin{equation*}
			L_n H(\cdot,s)+\sum_{k=1}^{l} \gamma_k(t)\mathcal{C}_k(H(\cdot,s))=0,
		\end{equation*}
		for all $t \in I$, $t \neq s$, $t \notin D$.
		\item For each $t \in \mathring{I}$, $t \notin D$, the lateral limits
		\begin{equation*}
			\frac{\partial^{n-1}}{\partial{t}^{n-1}}H(t^-,t)=\frac{\partial^{n-1}}{\partial t^{n-1}}H(t,t^+) \textup{ and }\frac{\partial^{n-1}}{\partial t^{n-1}}H(t,t^-)=\frac{\partial^{n-1}}{\partial t^{n-1}}H(t^+,t),
		\end{equation*}
		exist and are real. Furthermore,
		\begin{equation*}
			\frac{\partial^{n-1}}{\partial t^{n-1}}H(t^+,t)-\frac{\partial^{n-1}}{\partial t^{n-1}}H(t^-,t)=\frac{\partial^{n-1}}{\partial t^{n-1}}H(t,t^-)-\frac{\partial^{n-1}}{\partial t^{n-1}}H(t,t^+)=1.
		\end{equation*}
		\item For each $s \in \mathring{I}$, the function $H(\cdot,s)$ satisfies the boundary conditions
		\begin{equation*}
			V_i(H(\cdot,s))=0, \quad i=1, \ldots, n.
		\end{equation*}
	\end{itemize}
	\end{proposition}
\begin{remark}
	We notice that to prove Theorem \ref{teoprincipal}, we need to assume that there is a unique Green's function related to Problem \eqref{partir}. Nonetheless, it is worth noting that in some cases, Problem \eqref{proprin} may admit a unique solution, whereas Problem \eqref{partir} does not. As an illustration, let us consider the following example.
	\begin{equation}
		\left\{
		\begin{aligned}
			u'(t)+m\,u(t) +M\,u(0) &= \sigma(t), \quad && t \in [0, 1], \\
			u(0) &= u(1),
		\end{aligned}
		\right.
		\label{contra}
	\end{equation}
	where $m$, $M \in \mathbb{R}$ and $\sigma \in L^{1}([0,1])$.
	In \cite{cabada2014greens}, it is proved that the Green's function of Problem \eqref{contra} with $M=0$, which we denote by $G_{m}$, is given (when $m \neq 0$) by
	\begin{equation}
		G_{m}(t,s)=\frac{1}{e^{m}-1}\left\{
		\begin{aligned}
			e^{m(s-t+1)}, \quad &\textup{if }0 \leq s<t \leq 1,  \\
			e^{m(s-t)}, \quad &\textup{if }0 \leq t<s \leq 1.
		\end{aligned}
		\right.
		\label{defg}
	\end{equation}
	In the case when $m=0$, the homogeneous problem related to Problem \eqref{contra} with $M=0$ has nontrivial solutions. 
	Nevertheless, it is not difficult to verify that Problem \eqref{contra} admits a unique solution provided that $m+M \neq 0$. Taking into account Theorem \ref{teoprincipal} and using equation \eqref{hsimple}, we see that, whenever $m \neq 0$ and $m+M \neq 0$, the Green's function for Problem \eqref{contra} is given by
	\begin{equation*}
		H_{m,M}(t,s)=G_{m}(t,s)-\frac{M}{m+M}G_{m}(0,s),
	\end{equation*}
	where $G_{m}$ is defined in \eqref{defg}.
	
	When $m=0$, we see that $G_{0}$ is not well defined. Despite this, by direct integration, we obtain, for $M \neq 0$ that
	\begin{equation*}
		H_{0,M}(t,s) = 
		\begin{cases}
			\frac{1}{M}+1-t, & \text{if } 0 \leq s \leq t \leq 1, \\
			\frac{1 - M t}{M}, & \text{if } 0 \leq t < s \leq 1.
		\end{cases}
	\end{equation*}
	Thus, in this case, the Green's function corresponding to $m=0$ and $M \neq 0$ is uniquely determined and well-defined.
\end{remark}
We now present a series of examples that illustrate how studying problems of the form \eqref{proprin} allows us to analyze, among others, problems with integral-type dependence and those with piecewise constant arguments.
\begin{example}
	If the operator $\mathcal{C}$ is given by $\mathcal{C}(u)=u(1/2)$ and $\gamma(t)=1$, then
	\begin{equation*}
		H(t,s)=G(t,s)-G(1/2,s) \frac{\int_{a}^{b}{G(t,r)}\mathrm{d}r}{1+\int_{a}^{b}{G(1/2,r) \mathrm{d}r}}.
		\label{hsimple2}
	\end{equation*}
	As a representative example, we take
	\begin{equation}
		u'(t)+u(t)+u(1/2)=t, \quad t \in [-2,2], \quad u(-2)=0.
		\label{simedio}
	\end{equation}
	In Figure \ref{e1} we depict the Green's function $H$, while in Figure \ref{remedio} the solution and its derivative are shown. The explicit expression of $H$ is given in Appendix \ref{apsimedio}.
\begin{figure}[H]
	\centering
	\begin{minipage}[b]{0.48\textwidth}
		\centering
		\includegraphics[width=0.9\linewidth]{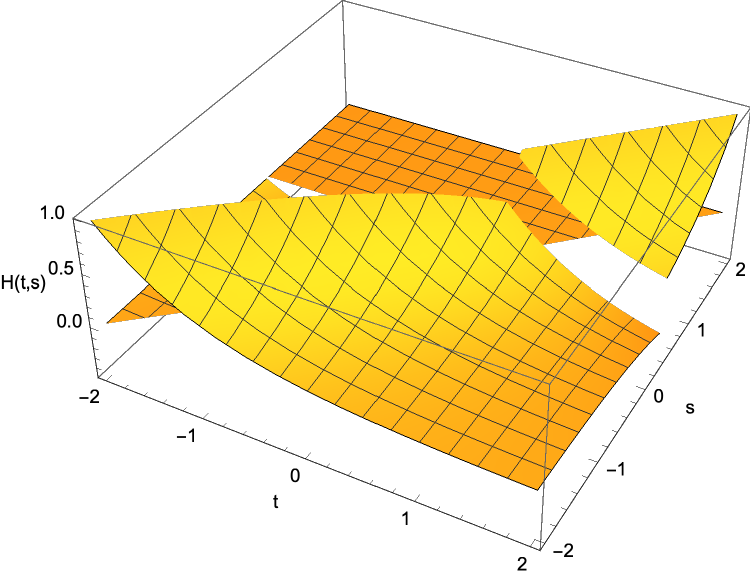}
	\end{minipage}\hfill
	\begin{minipage}[b]{0.48\textwidth}
		\centering
		\includegraphics[width=0.9\linewidth, trim={0 0 0 0.15cm}, clip]{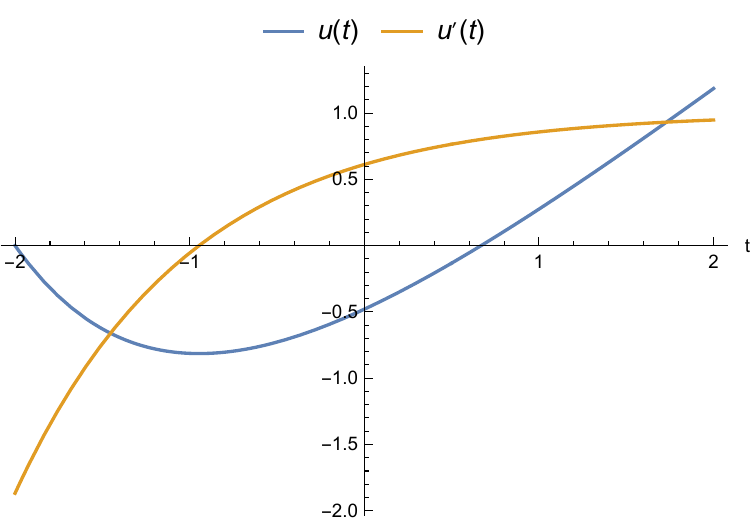}
	\end{minipage}

	\begin{minipage}[t]{0.48\textwidth}
		\caption{Representation of the Green's function $H$ of Problem \eqref{simedio}.}
		\label{e1}
	\end{minipage}\hfill
	\begin{minipage}[t]{0.48\textwidth}
		\caption{Representation of the solution $u$ and its derivative $u'$ corresponding to Problem \eqref{simedio}.}
		\label{remedio}
	\end{minipage}
\end{figure}
\end{example}
\begin{example}	 
\label{ej36}
	 Let us consider the case where $I=[-T,T]$ with $T>0$, and define
	 $\mathcal{C}_{j}(u)=u(j)$ for $j=[-T], \ldots, 0, \ldots, [T]$, where the function $[\cdot]$ (the integer part towards zero) is defined as
	 \begin{equation}
	 	\label{parteenteira}
	 	[t] = 
	 	\begin{cases}
	 		n, & \text{if } t \in [n, n+1), \\
	 		-n, & \text{if } t \in (-n-1, -n],
	 	\end{cases}
	 	\qquad n \in \mathbb{N}\cup\{0\},
	 \end{equation}
	 and $\gamma_j=M\chi_{J_j}$, $j=[-T],\ldots,[T]$, where
\begin{equation*}
	J_j=(j-1,j)\cap I \ \text{ if } j<0, \qquad
	J_0=(-1,1)\cap I, \qquad
	J_j=(j,j+1)\cap I \ \text{ if } j>0,
\end{equation*}
$M \in \mathbb{R}$ and $\chi_{J_j}(t)$ denotes the characteristic function on $J_j$.
	 
	 Then we obtain the following problem
	 \begin{equation}
	 	\left\{
	 	\begin{aligned}
	 		L_nu(t)+M\,u([t]) &= \sigma(t), \quad && t \in I, \\
	 		V_i(u) &= 0, \quad && i = 1, \ldots, n,
	 	\end{aligned}
	 	\right.
	 	\label{retardo}
	 \end{equation}
	where $V_{i}$ and $L_{n}$ are defined in \eqref{defv} and \eqref{defl}, respectively, and $\sigma \in L^{1}(I)$. 
	By applying Theorem \ref{teoprincipal}, we recover the formulas obtained in \cite[Section 3]{cabada2025reflection}.
	
	Our method also covers dependencies such as $u([t-j])$ or $u([t/2])$, which appear, for instance, in the models of \cite{karakoc2018oscillation, liu1999global}.
	
	\label{exaenteiro}	
\end{example}
\begin{example}
\label{ex37}
	Let us consider the same conditions as in the previous example, but now replace the linear operators $ \mathcal{C}_{j} $ by
	\begin{equation*}
	\mathcal{C}_j(u)=\int_{J_j}u(s)\,\mathrm{d}s, \qquad j=[-T],\ldots,[T],
\end{equation*}
where $J_j$ is defined in Example \ref{ej36}.
	Then, by applying again Theorem \ref{teoprincipal}, we can obtain the Green's function of the following problem:
	\begin{equation}
		\left\{
		\begin{aligned}
			L_n\,u(t) \;+\; M \int_{J_{[t]}} u(s)\, 	\mathrm{d}s &= \sigma(t), 
			&& t \in I, \\
			V_i(u) &= 0, 
			&& i = 1, \ldots, n,
		\end{aligned}
		\right.
		\label{pint}
	\end{equation}
 Moreover, $V_{i}$ and $L_{n}$ are defined in \eqref{defv} and \eqref{defl}, respectively, $\sigma \in L^{1}(I)$ and $J_{[t]}$ is the interval of the partition that contains $t$.
	
	As an illustrative example, we will consider the following problem:
	\begin{equation}
		u'(t)+u(t)+\int_{J_{[t]}}{u(s) \mathrm{d}s}=t, \, t \in [-2,2], \quad u(-2)=0,
		\label{ejint}
	\end{equation}
	with $J_{[t]}=(-2,-1)$ if $t \leq -1$, $(-1,1)$ if $|t|<1$, $(1,2)$ if $t \geq 1$.
	Based on the preceding discussion, we can obtain the corresponding Green's function and the explicit solution to Problem \eqref{ejint}. In Figure \ref{e2} we show its Green's function and in Figure \ref{repint}, we display the solution and its derivative.  The explicit expression of $H$ is given in Appendix \ref{apsimedio2}.

\begin{figure}[H]
	\centering
	\begin{minipage}[b]{0.48\textwidth}
		\centering
		\includegraphics[width=0.9\linewidth]{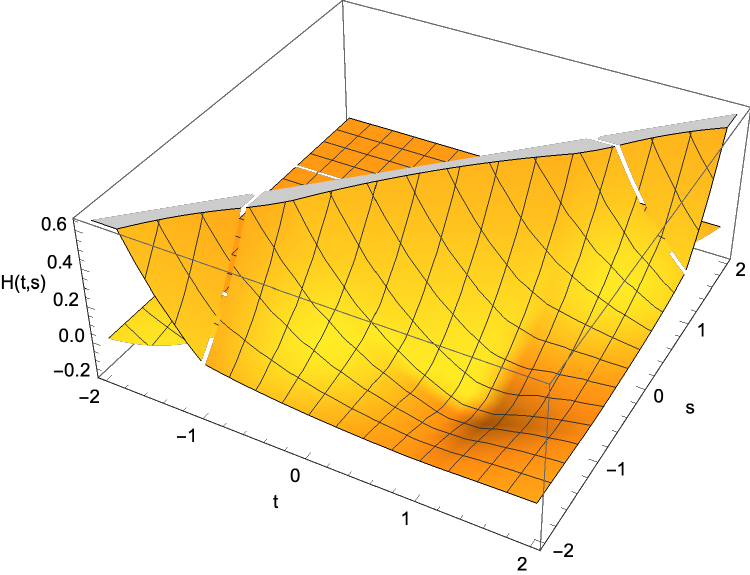}
	\end{minipage}\hfill
	\begin{minipage}[b]{0.48\textwidth}
		\centering
		\includegraphics[width=0.9\linewidth, trim={0 0 0 0.15cm}, clip]{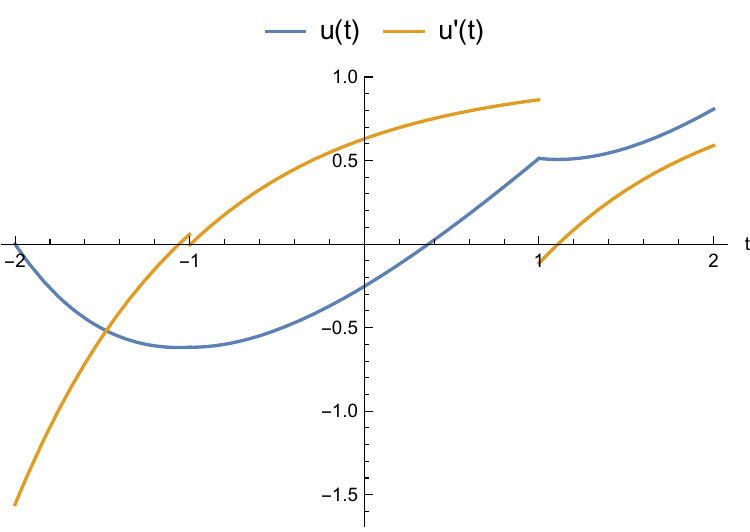}
	\end{minipage}

	\begin{minipage}[t]{0.48\textwidth}
		\caption{Representation of the Green's function $H$ of Problem \eqref{ejint}.}
		\label{e2}
	\end{minipage}\hfill
	\begin{minipage}[t]{0.48\textwidth}
		\caption{Representation of the solution $u$ and its derivative $u'$ corresponding to Problem \eqref{ejint}.}
		\label{repint}
	\end{minipage}
\end{figure}
\end{example}	
\begin{example}
	Our approach also allows us to deal with certain types of integro-differential equations, among which are the Fredholm integro-differential equations, of the general form
	\begin{equation}
		\left\{
		\begin{aligned}
			L_n u(t)+\int_a^b{K(t,s)u(s) \mathrm{d}s}&= \sigma(t), \quad && t \in I, \\
			V_i(u) &= 0, \quad && i = 1, \ldots, n,
		\end{aligned}
		\right.
		\label{fredgen}
	\end{equation}
	where $V_{i}$ and $L_{n}$ are defined in \eqref{defv} and \eqref{defl}, respectively, $\sigma \in L^{1}(I)$ and $K(t,s)$ is a kernel function modeling the influence of the values of $u$ on the whole interval $I$ on its evolution at time $t$. This integral term introduces a nonlocal dependence on the state of the function.

	In particular, we focus here on the case of separable (degenerate) kernels, namely those of the form
	\begin{equation}
		K(t,s)=\sum_{k=1}^{l} \phi_k(t) \xi_k(s),
		\label{expk}
	\end{equation}
	so that the integral operator reduces to
	\begin{equation*}
		\int_a^b{K(t,s)u(s)\mathrm{d}s}=\sum_{k=1}^{l}{\phi_k(t)} \left(\int_a^b{\xi_k(s)u(s) \mathrm{d}s} \right).
	\end{equation*}
	
Such equations appear in viscoelasticity, fluid dynamics with memory and control \cite{fabrizio1992mathematical,douglas1997single, bensoussan2007representation, polyanin2008handbook}. Moreover, separable kernels are important in themselves and also as approximations of general kernels \cite{wazwaz2011linear, yuldashev2018nonlocal}.

	If we consider $\mathcal{C}_k(u)=\int_a^b{\xi_k(s)u(s) \mathrm{d}s}$ and $\gamma_k(t)=\phi_k(t)$, for $k=1, \ldots, l$, the Green's function of Problem \eqref{fredgen} can be obtained via Theorem \ref{teoprincipal}.
	
	As an illustrative example, we consider
	\begin{equation}
		u'(t) + u(t) + \int_{-2}^{2} \left( \sin(t)\cos(s) + t^{2} e^{s} \right) 	u(s)\, \mathrm{d}s 
		= t, \quad t \in [-2,2], \quad u(-2)=0.
		\label{fred}
	\end{equation}
	In this case, we have $K(t,s)=\sin{(t)} \cos({s})+t^2e^s$ and
	\begin{equation*}
		\phi_{1}(t)=\sin{(t)}, \quad \phi_2(t)=\frac{3}{5}t^2, \quad \xi_1(s)=\cos{(s)} \quad \textup{and} \quad \xi_2(s)=e^s.
	\end{equation*}
	Here, the matrix $A$ given by \eqref{defa} satisfies $\det(A+I_2) \approx -0.0339$. Therefore, the problem is close to resonance, which explains the size of the solution in Figure \ref{grafred}. In Figure \ref{e3} we depict the Green's function, while Figure \ref{grafred} displays the solution to Problem \eqref{fred} along with its derivative.  The explicit expression of $H$ is given in Appendix \ref{apsimedio3}.
\begin{figure}[H]
	\centering
	\begin{minipage}[b]{0.48\textwidth}
		\centering
		\includegraphics[width=0.9\linewidth]{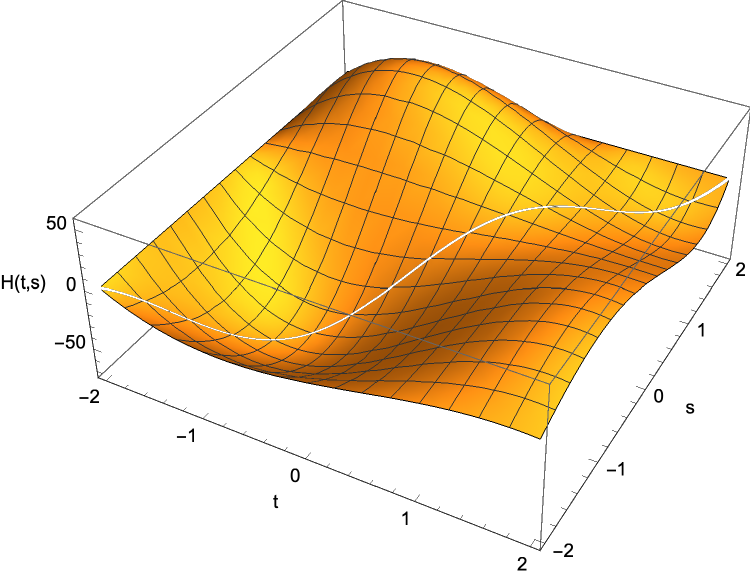}
	\end{minipage}\hfill
	\begin{minipage}[b]{0.48\textwidth}
		\centering
		\includegraphics[width=0.9\linewidth, trim={0 0 0 0.15cm}, clip]{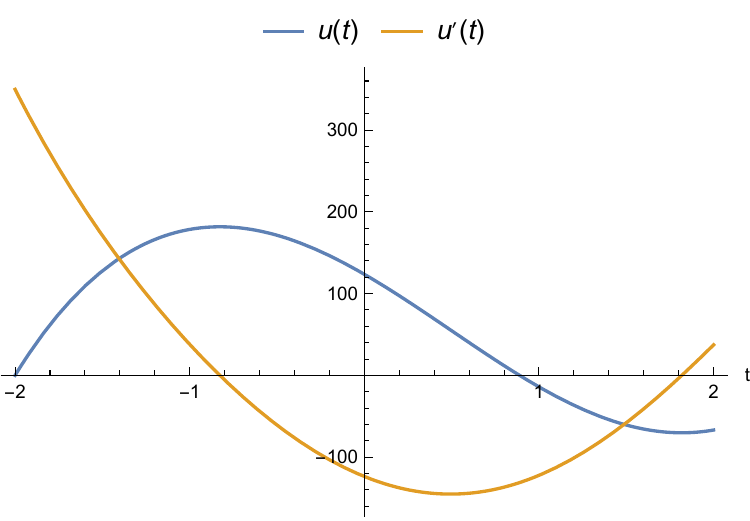}
	\end{minipage}

	\begin{minipage}[t]{0.48\textwidth}
		\caption{Representation of the Green's function $H$ of Problem \eqref{fred}.}
		\label{e3}
	\end{minipage}\hfill
	\begin{minipage}[t]{0.48\textwidth}
		\caption{Representation of the solution $u$ and its derivative $u'$ corresponding to Problem \eqref{fred}.}
		\label{grafred}
	\end{minipage}
\end{figure}
\end{example}
Note that these examples help us verify the properties of the Green's functions stated in Proposition \ref{caracterizarh}. In particular, in Figure \ref{e1} we observe that $H(\cdot,s)$ is of class $C^1$ except at $t=s$, where it has a jump, whereas $H(t,\cdot)$ exhibits a jump discontinuity at $s=t$ and another at $s=\frac{1}{2}$.

On the other hand, in Figure \ref{e2}, we see that $H\left(\cdot,s\right)$ has a jump at $t=s$ and is not differentiable at $t=-1$ and $t=1$, analogously to $H\left(t,\cdot\right)$, which has a jump at $s=t$ and is not differentiable at $s=-1$ and $s=1$.

Finally, in Figure \ref{e3} we observe that both $H(t,\cdot)$ and $H(\cdot,s)$ are of class $C^1$ except at $t=s$.

\section{Comparison Results}
\label{comparison}
In this section, we establish a relationship between the Green's functions related to two problems of the form \eqref{proprin}, sharing the same linear operators but involving different functions $\gamma_{i}(t)$. Our approach follows the ideas presented in \cite{cabada2024explicit}.
\begin{remark}
	Although in all the following sections we focus on functional problems of the form of Problem \eqref{proprin}, these results can also be extended to the corresponding adjoint problems, which will be obtained later on. To this end, the procedure can be carried out by taking into account that $G^*(t,s)=G(s,t)$, where by $G^*$ we denote the Green's function related to the adjoint problem.
\end{remark}
Consider the two following problems:
\begin{equation}
	L_{n}u_{0}(t)+\sum_{k=1}^{l}{\gamma_{k}^{0}(t)\mathcal{C}_{k}(u_{0})}=\sigma(t), \quad t \in I, \quad V_{i}(u_{0})=0,\, i=1, \ldots, n,
	\label{prob1}
\end{equation}
\begin{equation}
	L_{n}u_{1}(t)+\sum_{k=1}^{l}{\gamma_{k}^{1}(t)\mathcal{C}_{k}(u_{1})}=\sigma(t), \quad t \in I, \quad V_{i}(u_{1})=0, \, i=1, \ldots, n,
	\label{prob2}
\end{equation}
where $V_{i}$ and $L_{n}$ are defined in \eqref{defv} and \eqref{defl}, respectively, and $\sigma \in L^{1}(I)$. 

From the two previous expressions, we arrive at the following equality for all $t \in I$:
\begin{equation*}
	L_{n}u_{0}(t)+\sum_{k=1}^{l}{\gamma_{k}^{1}(t) \mathcal{C}_{k}(u_{0})}=\sum_{k=1}^{l}{\left(\gamma_{k}^{1}(t)-\gamma_{k}^{0}(t) \right)}\mathcal{C}_{k}(u_{0})+\sigma(t), \quad t \in I, \quad V_{i}(u_{0})=0, \, i=1, \ldots, n.
\end{equation*}
Let $H_{0}$ denote the Green's function related to Problem \eqref{prob1} and $H_{1}$ denote the Green's function of Problem \eqref{prob2}, assuming that they exist. Consequently
\begin{equation*}
	u_{0}(t)=\int_{a}^{b}{H_{0}(t,s) \sigma(s) \mathrm{d}s} \quad \text{and} \quad u_{1}(t)=\int_{a}^{b}{H_{1}(t,s) \sigma(s) \mathrm{d}s}.
\end{equation*}
Hence, we obtain that
\begin{equation*}
	\begin{aligned}
		u_{0}(t)&=\int_{a}^{b}{H_{1}(t,s)\left(\sum_{k=1}^{l}(\gamma_{k}^{1}(s)- \gamma_{k}^{0}(s))\mathcal{C}_{k}(u_{0}) \right) \mathrm{d}s}+\int_{a}^{b}{H_{1}(t,s) \sigma(s) \mathrm{d}s} \\
		&=\int_{a}^{b}{H_{1}(t,s) \left(\sum_{k=1}^{l}\mathcal{C}_{k}\left(\int_{a}^{b}{H_{0}(\cdot,r)\sigma(r)}\mathrm{d}r\right) (\gamma_{k}^{1}(s)-\gamma_{k}^{0}(s))\right) \mathrm{d}s}+\int_{a}^{b}{H_{1}(t,s) \sigma(s) \mathrm{d}s} \\
		&=\int_{a}^{b}\int_{a}^{b}H_{1}(t,r)\left( \sum_{k=1}^{l}{(\gamma_{k}^{1}(r)-\gamma_{k}^{0}(r))}\mathcal{C}_{k}(H_{0}(\cdot,s)) \sigma(s) \right) \mathrm{d}r \,\mathrm{d}s+\int_{a}^{b}{H_{1}(t,s) \sigma(s) \mathrm{d}s}.
	\end{aligned}
\end{equation*}
Consequently, as the previous equalities hold for every $\sigma \in L^{1}(I)$, we ultimately conclude that
\begin{equation}
	H_{0}(t,s)=H_{1}(t,s)+\sum_{k=1}^{l}{\mathcal{C}_{k}(H_{0}(\cdot,s))\int_{a}^{b}(\gamma_{k}^{1}(r)-\gamma_{k}^{0}(r))H_{1}(t,r) \mathrm{d}r}.
	\label{eccom}
\end{equation}
Furthermore, in an analogous way, we can see
\begin{equation}
	H_{1}(t,s)=H_{0}(t,s)-\sum_{k=1}^{l}{\mathcal{C}_{k}(H_{1}(\cdot,s))\int_{a}^{b}(\gamma_{k}^{1}(r)-\gamma_{k}^{0}(r))H_{0}(t,r) \mathrm{d}r}.
	\label{eccom2}
\end{equation}
\begin{example}
	If we consider two problems of the same form as in Example \ref{exaenteiro} (with piecewise constant arguments), the first with $M=M_{0}$ and the second with $M=M_{1}$, then the relationship between the corresponding Green's functions, denoted by $H_{M_{0}}$ and $H_{M_{1}}$, is given by the following equation:
	\begin{equation*}
		H_{M_{0}}(t,s)=H_{M_{1}}(t,s)+(M_{1}-M_{0})\int_{-T}^{T}{H_{M_{1}}(t,r)H_{M_{0}}([r],s) \mathrm{d}r}.
	\end{equation*}
	This yields the same equality as that of equation $(4)$ in \cite{cabada2025reflection} with $m_1=m_0=0$.
\end{example}
\begin{proposition}
\label{propsigno}
Let $\gamma_k^0$, $\gamma_k^1 \in L^1(I)$ with $\gamma_k^0 \leq \gamma_k^1$ a.e. on $I$, $k=1, \ldots, l$, and let $H_0$, $H_1$ be the Green's functions of \eqref{prob1} and \eqref{prob2}. Assume that one of the following conditions holds on $I \times I$ for every $k=1, \ldots, l$ and a.e.\ $s \in I$:
\begin{enumerate}
\item $H_1 \geq 0$ and $\mathcal{C}_k(H_0(\cdot,s)) \geq 0$.
\item $H_1 \leq 0$ and $\mathcal{C}_k(H_0(\cdot,s)) \leq 0$.
\item $H_0 \geq 0$ and $\mathcal{C}_k(H_1(\cdot,s)) \geq 0$.
\item $H_0 \leq 0$ and $\mathcal{C}_k(H_1(\cdot,s)) \leq 0$.
\end{enumerate}
Then $H_1 \leq H_0$ on $I \times I$. 

If the inequalities are strict on $\mathring{I} \times \mathring{I}$ and $\gamma_k^1 > \gamma_k^0$ on a set of positive measure for some $k$, then $H_1<H_0$ on $\mathring{I} \times \mathring{I}$. 
\end{proposition}
\begin{proof}
By \eqref{eccom} and \eqref{eccom2}
\begin{equation*}
\begin{aligned}
H_0(t,s)-H_1(t,s)
&=\sum_{k=1}^{l}\mathcal{C}_k(H_0(\cdot,s))\int_a^b
  \bigl(\gamma^1_k(r)-\gamma^0_k(r)\bigr)H_1(t,r)\,\mathrm{d}r\\
&=\sum_{k=1}^{l}\mathcal{C}_k(H_1(\cdot,s))\int_a^b
  \bigl(\gamma^1_k(r)-\gamma^0_k(r)\bigr)H_0(t,r)\,\mathrm{d}r .
\end{aligned}
\end{equation*}
Assumptions $1$ and $2$ ensure that each summand in the first sum is nonnegative, whereas $3$ and $4$ yield the same conclusion for the second sum. Thus, $H_0 \geq H_1$. If the inequalities are strict and $\gamma_k^1> \gamma_k^0$ on a set of positive measure for some $k$, the corresponding summand is positive on $\mathring{I} \times \mathring{I}$, which gives $H_0>H_1$ there. 
\end{proof}

\section{Adjoint Operator}
\label{adjoint}
The aim of this section is to obtain the adjoint operator associated with the previously formulated problem and to characterize the corresponding Green's function.

This can be of particular interest because, in light of the results presented below, knowing the sign of the Green's function related to a given operator also determines the sign of the Green's function of its adjoint. Consequently, this allows us to study simultaneously the existence of solutions for a problem and for its adjoint.

Throughout this section we assume $a_j \in C^{n-j}(I)$, $j=1, \ldots, n$, so that the formal adjoint 
\begin{equation*}
L_n^*v=\sum_{j=0}^n (-1)^j (a_{n-j}v)^{(j)}, \qquad \text{with }a_0=1
\end{equation*} 
and the adjoint boundary operators $V_i^*$ are well defined \cite[Section~1.4]{cabada2014greens}. We also assume that the functionals $\mathcal{C}_k$ do not involve $u^{(n)}$.

Let us now recall Theorem 1.3.1 of \cite{cabada2014greens}. 
\begin{theorem}
	Let $\mathcal{G}: L^2(I, \mathbb{R}^{n}) \rightarrow L^{2}(I, \mathbb{R}^{n})$ be an operator given by
	\begin{equation}
		\mathcal{G} u(t)=\int_a^b{K(t,s)u(s) \mathrm{d}s}, \quad t \in I,
	\end{equation}
	where $K: I \times I \rightarrow \mathcal{M}_{n \times n}$, such that $K \in L^2(I \times I, \mathcal{M}_{n \times n})$. If we denote by $K^T$ the transpose matrix kernel of $K$, then, for all $v \in L^2(I, \mathbb{R}^n)$, the adjoint operator of $\mathcal{G}$, $\mathcal{G}^{*}$, is given by the following expression.
	\begin{equation}
		\mathcal{G}^{*}v(t)=\int_a^b{K^T(s,t)v(s)\mathrm{d}s}, \quad t \in I.
	\end{equation}
\end{theorem}
As a consequence of the previous theorem, we obtain the following result.
\begin{corollary}
	\label{coroadjoint}
	Let $G$ be the Green's function for the boundary value problem $L_nu=\sigma$ subject to the boundary conditions $V_i (u)=0$. Then the Green's function $G^*$ for the adjoint problem $L_n^*v=\sigma^{*}$ with adjoint boundary conditions $V_i^*(v)=0$ exists and is given by
	\begin{equation*}
		G^*(t,s)=G(s,t).
	\end{equation*}
\end{corollary}
To begin, we consider a slight modification of Problem \eqref{proprin}. Specifically, we work with operators $\overline{\mathcal{C}}_k:\mathcal{H} \rightarrow \mathbb{R}$, where $\mathcal{H}=L^2(I)$, $\gamma_k \in L^2(I)$, and the operators $\overline{\mathcal{C}}_k$ are linear and continuous with respect to the norm of $\mathcal{H}$. Since $W^{n,1}(I)\hookrightarrow C^{n-1}(I)\subset L^\infty(I)$ continuosly and $\|u\|_2 \leq (b-a)^{1/2}\|u\|_{\infty}$, the restrictions of such $\overline{\mathcal{C}}_k$ to $W^{n,1}(I)$ are continuous and, since $\gamma_k\in L^2(I)\subset L^1(I)$, this is a particular case of the previous setting. However, functionals such as $u \rightarrow u^{(\alpha)}(t_k)$ are not covered, since they are not continuous in $L^2(I)$.

Therefore, we consider the following problem:
\begin{equation}
	\left\{
	\begin{aligned}
		K_n u(t) := L_n u(t) +\sum_{k=1}^l \gamma_k(t) \overline{\mathcal{C}}_k(u)&= 0, \quad && t \in I, \\
		V_i(u) &= 0, \quad && i = 1, \ldots, n,
	\end{aligned}
	\right.
	\label{modificado}
\end{equation}
where $V_{i}$ and $L_{n}$ are defined in \eqref{defv} and \eqref{defl}, respectively.

In line with this objective, we arrive at the following result:
\begin{theorem}
	\label{teoadxunto}
	The adjoint boundary value problem corresponding to \eqref{modificado} can likewise be represented as a functional problem of the same type and is explicitly given by
	\begin{equation}
		\left\{
		\begin{aligned}
			K_n^{*} v(t) := L_n^{*} v(t)+\sum_{k=1}^{l}\eta_{k}(t)\mathcal{D}_{k}(v)&=0, \quad && t \in I, \\
			V_i^{*}(v) &= 0, \quad && i = 1, \ldots, n,
		\end{aligned}
		\right.
		\label{oadjunto}
	\end{equation}
	where
	\begin{equation}
		\left\{
		\begin{aligned}
			L_n^{*}v(t)&=0, \quad && t \in I, \\
			V_i^{*}(v)&= 0, \quad && i = 1, \ldots, n,
		\end{aligned}
		\right.
		\label{oadjuntosim}
	\end{equation}
	is the adjoint of Problem \eqref{partir}, and the functionals $\mathcal{D}_{i}: \mathcal{H} \rightarrow \mathbb{R}$ and $\eta_{i} \in \mathcal{H}=L^2(I)$, $i=1, \ldots, l$, are deduced from the following expressions:
	\begin{equation}
		\label{defdi}
		\mathcal{D}_{i}(u)=\langle\gamma_{i},u \rangle
	\end{equation}
	and
	\begin{equation}
		\label{defetai}
		\overline{\mathcal{C}}_{i}(u)=\langle u,\eta_{i} \rangle,
	\end{equation}
	where $\langle \cdot,\cdot \rangle$ is the scalar product in $\mathcal{H}$.
\end{theorem}
\begin{proof}
	We will prove that $K_n^{*}$ is the adjoint of operator $K_{n}$. To this aim, we will see that
	\begin{equation*}
		\langle K_n u,v \rangle = \langle u, K_n^{*} v \rangle, \quad \forall u,v \in \mathcal{H} \textup{ such that }V_i(u)=0, V_i^*(v)=0, \, i=1, \ldots, n.
	\end{equation*}
	Using the definition of $\mathcal{D}_i$ given in \eqref{defdi} and the definition of $\eta_i$ given in \eqref{defetai}, whose existence is ensured by the Riesz-Fréchet Representation Theorem, we obtain 
	\begin{equation}
		\langle \gamma_{k}(\cdot) \overline{\mathcal{C}}_k(u), v \rangle =\langle u, \eta_k(\cdot) \mathcal{D}_k(v)\rangle, \quad \forall \, k=1, \ldots, l.
		\label{igualdadcd}
	\end{equation}	
	Summing over $k=1, \ldots, l$, we obtain
	\begin{equation*}
		\sum_{k=1}^{l}{\langle \gamma_{k}(\cdot) \overline{\mathcal{C}}_k(u), v \rangle}=\sum_{k=1}^{l}{\langle u, \eta_k(\cdot) \mathcal{D}_k(v)\rangle }.
	\end{equation*}
	On the other hand, it holds that
	\begin{equation*}
		\langle L_n u, v \rangle= \langle u, L_n^* v\rangle.
	\end{equation*}
	Finally, we obtain
	\begin{equation*}
		\langle L_n u+\sum_{k=1}^{l} \gamma_k(\cdot)\overline{\mathcal{C}}_k(u), v \rangle = \langle u, L_n^{*}v+\sum_{k=1}^{l}{\eta_k(\cdot)\mathcal{D}_k(v)} \rangle.
	\end{equation*}
	Since $Bu:=\sum_{k=1}^l \gamma_k \mathcal{\overline{C}}_k(u)$ is a bounded operator on $L^2(I)$ with $B^*v=\sum_{k=1}^{l}\eta_k \mathcal{D}_k(v)$, we have $(L_n+B)^*=L_n^*+B^*$, with the same domain as $L_n^*$. Hence the adjoint is given by \eqref{oadjunto}.
\end{proof}
\begin{remark}
	It is interesting to note that $\mathcal{D}_k(u)$ is the distribution generated by the function $\gamma_k$, and that $\eta_k$ coincides with the generating function of the distribution $\overline{\mathcal{C}}_k(u)$.
\end{remark}
\begin{remark}
	\label{remarkadjoint}
	Assume that $\gamma_1, \ldots, \gamma_l$ are linearly independent. Then the functional part is self-adjoint if and only if the functionals $\overline{\mathcal{C}}_k$ satisfy the following condition:
	\begin{equation*}
		\overline{\mathcal{C}}_k(u)=\sum_{j=1}^l\beta_{kj} \langle \gamma_j,u \rangle, \quad k=1, \ldots, l,
	\end{equation*}
	for some choice of $\beta_{kj} \in \mathbb{R}$ such that $\beta_{kj}=\beta_{jk}$ for all $k=1, \ldots, l$, $j=1, \ldots, l$.

	Indeed, self-adjointness of the functional part means that, for all $u,v \in \mathcal{H}$,
	\begin{equation*}
		\sum_{k=1}^l \overline{\mathcal{C}}_k(u) \langle \gamma_k,v \rangle= \sum_{j=1}^l \overline{\mathcal{C}}_j(v) \langle \gamma_j, u \rangle.
	\end{equation*}
	Defining $\Gamma_k(f)=\langle \gamma_k,f \rangle$, we have that
	\begin{equation*}
		\sum_{k=1}^l\overline{\mathcal{C}}_k(u) \Gamma_k(v)=\sum_{j=1}^l\overline{\mathcal{C}}_j(v) \Gamma_j(u).
	\end{equation*}
	Since the functions $\gamma_k$ are linearly independent, the functionals $\Gamma_k$ are also linearly independent.
	
	Taking $u=u_m$ with $\Gamma_j(u_m)=\delta_{jm}$ (possible since the $\Gamma_j$ are linearly independent), we obtain $\overline{\mathcal{C}}_m(v)=\sum_{k=1}^l \overline{\mathcal{C}}_k(u_m)\Gamma_k(v)$, that is,
	\begin{equation*}
		\overline{\mathcal{C}}_k(v)=\sum_{j=1}^l \beta_{kj} \Gamma_{j}(v)=\sum_{j=1}^l \beta_{kj} \langle \gamma_j,v \rangle.
	\end{equation*}
	In this case,
	\begin{equation*}
		\sum_{k=1}^l \sum_{j=1}^l \beta_{kj} \Gamma_j(u) \Gamma_k(v)=\sum_{j=1}^l \sum_{k=1}^l \beta_{jk} \Gamma_{k}(v) \Gamma_{j}(u),
	\end{equation*}
	from which $\beta_{kj}=\beta_{jk}$ for all $j,k \in \{1, \ldots,l\}$.
\end{remark}
Theorem \ref{teoadxunto} allows us to refine Proposition \ref{caracterizarh} by providing properties of $H(t, \cdot)$ in the case where the operators $\mathcal{C}_j: \mathcal{H} \rightarrow \mathbb{R}$ are linear and continuous with respect to the norm of a Hilbert space for all $j=1, \ldots, l$.

To this end, we define the following set:
\begin{equation*}
	D^*=\{t \in I: \eta_i(t) \textup{ is discontinuous for some }i \in \{1, \ldots, l\}\}.
\end{equation*}
\begin{proposition}
	Assume that $\eta_1, \ldots, \eta_l$ are piecewise continuous. Then, the Green's function $H$ of Problem \eqref{proprin}, when $\mathcal{C}_j:\mathcal{H} \rightarrow \mathbb{R}$ are linear and continuous in $\mathcal{H}$ for all $j=1, \ldots, l$, also satisfies the following properties:
	\begin{itemize}
		\item $H(t, \cdot)$ is well-defined and (if $n \geq 2$) of class $C^{n-2}$ for all $s \in I$, of class $C^{n-1}$ for all $s \in I$, $s \neq t$, and of class $C^n$ for all $s \in I$, $s \neq t$, $s \notin D^*$.
		\item For each $t \in \mathring{I}$, the function $H(t,\cdot)$ is the solution of the following differential equation:
		\begin{equation*}
			L_n^* H(t,\cdot)+\sum_{k=1}^{l} \eta_k(s)\mathcal{D}_k(H(t,\cdot))=0,
		\end{equation*}
		for all $s \in I$, $s \neq t$, $s \notin D^*$, where $\eta_k$ and $\mathcal{D}_k$ are given in Theorem \ref{teoadxunto}.
		\item For each $t \in \mathring{I}$, the function $H(t,\cdot)$ satisfies the boundary condition:
		\begin{equation*}
			V_i^*(H(t,\cdot))=0, \quad i=1, \ldots, n.
		\end{equation*}
	\end{itemize}
\end{proposition}
As illustrative cases, we will consider the following examples:
\begin{example}
\label{ex58}
	We consider the integro-differential Problem \eqref{fredgen} with $K$ given by \eqref{expk}.
	From Theorem \ref{teoadxunto}, we have that
	\begin{equation*}
		\mathcal{D}_k(u)=\int_a^b{\phi_k(s)u(s) \mathrm{d}s} \, \textup{ and } \, \eta_k(t)=\xi_k(t), \, \textup{ for }\, k=1, \ldots, l.
	\end{equation*}
	Therefore, the adjoint problem of \eqref{fredgen} for a separable kernel is given by
	\begin{equation*}
		\left\{
		\begin{aligned}
			L^*_n v(t)+\int_a^b{K(s,t) v(s) \mathrm{d}s}&=0 , \quad && t \in I, \\
			V^*_i(v) &= 0, \quad && i = 1, \ldots, n.
		\end{aligned}
		\right.
		\label{integrodiferencialad}
	\end{equation*}
	See also \cite{Akhiev2007,Sirma2015} for adjoint systems and Green's functionals of integro-differential problems with nonlocal conditions.
\end{example}
\begin{example}
	We consider a perturbation of Hill's equation given by the following equation:
	\begin{equation}
		\left\{
		\begin{aligned}
			K_n u(t):=u''(t)+q(t)u(t)+2t\int_a^b{su(s) \mathrm{d}s}+3e^{t}\int_a^b {e^s u(s) \mathrm{d}s}+\int_a^bu(s) \mathrm{d}s&=0 , && t \in I , \\
			u(a) = u(b),\, u'(a)&=u'(b),  
		\end{aligned}
		\right.
		\label{hillperturbada}
	\end{equation}
	where $\phi_1(t)=2t$, $\xi_1(s)=s$, $\phi_2(t)=3e^t$, $\xi_2(s)=e^s$, $\phi_3(t)=1$ and $\xi_3(s)=1$.
	
	Taking into account that Hill's equation is self-adjoint and Remark \ref{remarkadjoint}, the operator in \eqref{hillperturbada} is also self-adjoint, since its kernel $2ts+3e^{t+s}+1$ is symmetric.
\end{example}
In Problem \eqref{proprin}, the operators $\mathcal{C}_k:W^{n,1}(I) \rightarrow \mathbb{R}$ were considered. Although Theorem \ref{teoadxunto} cannot be applied because $W^{n,1}(I)$ is not a Hilbert space, it is still possible and worthwhile to calculate the adjoint in certain cases.

As an interesting particular case, we have equations with piecewise constant dependence where the operators $\mathcal{C}_k^{\alpha}$ are given by $\mathcal{C}_k^{\alpha}(u)=u^{(\alpha)}(t_k)$, $t_k \in I$, for all $k=1, \ldots, l$, $\alpha=0, \ldots, n-1$. In this case, $\mathcal{C}_k^{\alpha}$ are not continuous with respect to the norm of $L^2{(I)}$. Therefore, we cannot apply the Riesz-Fréchet Theorem. 

Nevertheless, the adjoint can be calculated (for the general theory of adjoints of differential operators with interior-point or Stieltjes boundary conditions and of differential-boundary operators, see \cite{Krall1971,Krall1977,BrownKrall1974,Brown1975}). To this end, we recall the definition of the Dirac delta and its derivatives.
\begin{definition}
	\label{defdelta}
	Let $t_k\in I$ and $\alpha\in\{0,\ldots,n-1\}$. The Dirac delta centered at $t_k$ and its derivatives act on $\varphi\in C^{n-1}(I)$ by
	\[
	\left\langle \delta^{(\alpha)}(s-t_k),\varphi(s)\right\rangle
	=\int_a^b \varphi(s)\,\delta^{(\alpha)}(s-t_k)\, \mathrm{d}s
	:=(-1)^\alpha \varphi^{(\alpha)}(t_k),
	\]
	where $\delta^{(0)}=\delta$ and $\langle \cdot, \cdot \rangle$ denotes the duality pairing. If $a<t_k<b$, this is the usual action of the distribution $\delta^{(\alpha)}(s-t_k)$. If $t_k\in\{a,b\}$, the pairing is understood on the closed interval $I$, and $\varphi^{(\alpha)}(t_k)$ is the corresponding one-sided derivative.
\end{definition}
\begin{remark}
	From now on, $\langle \cdot, \cdot \rangle$ denotes the duality pairing, which in general does not coincide with the inner product in $L^2(I)$, denoted by $\langle \cdot, \cdot \rangle_2$.
\end{remark}

Let $\mathcal{T}:=\{t_1,\ldots,t_l\}$, with $a\leq t_1<\cdots<t_l\leq b$, and set $t_0:=a$ and $t_{l+1}:=b$. We define
\begin{equation*}
PC_{\mathcal{T}}^{n-1}(I)
:=
\left\{
v\in L^1(I):
v|_{(t_k,t_{k+1})}\in W^{n,1}(t_k,t_{k+1}),
\quad k=0,\ldots,l
\right\}.
\end{equation*}
Since $W^{n,1}(t_k,t_{k+1})\subset C^{n-1}([t_k,t_{k+1}])$, these functions are piecewise of class $C^{n-1}$ and the one-sided limits $v^{(r)}(t_k^\pm)$, $r=0,\ldots,n-1$, exist, hence the notation.

Since $u \rightarrow u^{(\alpha)}(t_k)$ is not bounded in $L^2(I)$, the adjoint below is understood in the sense $\langle K_nu,v \rangle_2=\langle u, K_n^* v\rangle_2$ for all $u \in D(K_n)$ and all $v \in PC_\mathcal{T}^{n-1}(I)$ satisfying \eqref{adxuns2compb}-\eqref{adxuns2compc} (see Definition \ref{defotra} and Proposition \ref{propotra} for an equivalent weak formulation).

\begin{lemma}
	\label{lemagreen}
	For every $u\in W^{n,1}(I)$ and every $v\in PC_{\mathcal{T}}^{n-1}(I)$,
	\begin{equation*}
		\begin{aligned}
			\int_a^b L_nu(t)\,v(t)\,\mathrm{d}t
			=
			\sum_{k=0}^{l}\int_{t_k}^{t_{k+1}}u(t)\,L_n^*v(t)\,\mathrm{d}t
			-\sum_{t_k\in(a,b)}\sum_{\alpha=0}^{n-1}\Lambda_k^\alpha(v)\,u^{(\alpha)}(t_k)
			+\sum_{\alpha=0}^{n-1}\Bigl(\mathcal{B}_b^\alpha(v)\,u^{(\alpha)}(b)-\mathcal{B}_a^\alpha(v)\,u^{(\alpha)}(a)\Bigr),
		\end{aligned}
	\end{equation*}
	where $L_n^*v$ is computed on each subinterval $(t_k,t_{k+1})$, $\Delta v^{(r)}(t_k)=v^{(r)}(t_k^+)-v^{(r)}(t_k^-)$,
	\begin{equation*}
		\begin{aligned}
			\Lambda_k^\alpha(v)&=\sum_{j=\alpha+1}^{n}\sum_{r=0}^{j-1-\alpha}(-1)^{j-1-\alpha}\binom{j-1-\alpha}{r}a_{n-j}^{(j-1-\alpha-r)}(t_k)\,\Delta v^{(r)}(t_k),\\
			\mathcal{B}_c^\alpha(v)&=\sum_{j=\alpha+1}^{n}(-1)^{j-1-\alpha}(a_{n-j}v)^{(j-1-\alpha)}(c), \qquad c\in\{a,b\},
		\end{aligned}
	\end{equation*}
	and the values of $v$ and its derivatives at $a$ and $b$ are one-sided limits.
\end{lemma}
\begin{proof}
	Since $v$ may have jumps at the points of $\mathcal{T}$, we decompose the integral at these points. As $u \in C^{n-1}(I)$, we have
	\begin{equation*}
		\langle L_n u,v \rangle_2=\int_a^b L_n u(t)v(t)\, \mathrm{d}t=\sum_{k=0}^{l}\int_{t_{k}}^{t_{k+1}}L_n u(t)v(t) \,\mathrm{d}t.
	\end{equation*}
	Next, following \cite[Section~1.4]{cabada2014greens}, we obtain:
	\begin{equation*}
		\int_c^d a_{n-j}(t)u^{(j)}(t)v(t)\, \mathrm{d}t=(-1)^j \int_c^d (a_{n-j}v)^{(j)}(t)u(t)\, \mathrm{d}t+\sum_{i=0}^{j-1}(-1)^{j-1-i}\Big[(a_{n-j}v)^{(j-1-i)}u^{(i)}\Big]_{c}^{d}.
	\end{equation*}
	For simplicity in the notation, we denote $a_0(t)=1$. Applying this identity on each subinterval $(t_k,t_{k+1})$ and summing over $j=0,\ldots,n$ and $k=0,\ldots,l$, we obtain that
	\begin{equation*}
		\int_a^b L_n u(t)v(t)\, \mathrm{d}t=\sum_{k=0}^l \int_{t_k}^{t_{k+1}} u(t)L_n^*v(t)\, \mathrm{d}t +\sum_{j=1}^{n}\sum_{k=0}^{l} \sum_{i=0}^{j-1}(-1)^{j-1-i}\Big[(a_{n-j}v)^{(j-1-i)}u^{(i)}\Big]_{t_{k}^+}^{t_{k+1}^-}.
	\end{equation*}
	Furthermore, using that
	\begin{equation*}
		(a_{n-j}v)^{(j-1-i)}(t)=\sum_{r=0}^{j-1-i} \binom{j-1-i}{r}a_{n-j}^{(j-1-i-r)}(t)\,v^{(r)}(t), \quad t \neq t_k,
	\end{equation*}
	we can rewrite the boundary contributions as follows. When summing over all the subintervals, the terms corresponding to each interior point $t_k$ combine through the difference between their left- and right-hand traces, whereas the contributions at the endpoints $a$ and $b$ remain separate. Since $u^{(i)}$ and the coefficients $a_{n-j}$ are continuous at $t_k$, the interior contributions can be expressed in terms of the jumps of $v$ and its derivatives, which gives the identity in the statement.
\end{proof}

With this in hand, we present the following result.
\begin{theorem}
	\label{adjpiece2}
	The adjoint equation of 
	\begin{equation}
		\left\{
		\begin{aligned}
			K_n u(t):=L_n u(t) +\sum_{k=1}^l \sum_{\alpha=0}^{n-1} \gamma_k^{\alpha}(t) u^{(\alpha)}(t_k)&=0, \quad && t \in I,  \\
			V_i(u) &= 0, \quad && i = 1, \ldots, n,
		\end{aligned}
		\right.
		\label{adxunscomp}
	\end{equation}
	where $V_{i}$ and $L_{n}$ are defined in \eqref{defv} and \eqref{defl}, respectively, $a\leq t_1 <t_2 < \cdots <t_l \leq b$ with $a_j \in C^{n-j}(I)$ and $a_0 \equiv 1$, is given by
	\begin{equation}
		L_n^* v(t)=0, \quad t \in I \setminus \{t_1, \ldots, t_l\},
		\label{adxuns2compa}
	\end{equation}
	with jump conditions for each $k=1, \ldots, l$ and $\alpha=0, \ldots,n-1$, $t_k \neq a$ and $t_k \neq b$
	\begin{equation}
		\sum_{j=\alpha+1}^{n} \sum_{r=0}^{j-1-\alpha} (-1)^{j-1-\alpha}
		\binom{j-1-\alpha}{r} a_{n-j}^{(j-1-\alpha-r)}(t_k) 
		\Delta v^{(r)}(t_k)
		= \int_a^b \gamma_k^\alpha(t)v(t)\,\mathrm{d}t,
		\label{adxuns2compb}
	\end{equation}
	and boundary conditions if $t_1 \neq a$ and $t_l \neq b$ given by
	\begin{equation*}
		V_i^*(v)=0, \quad i=1, \ldots, n,
	\end{equation*}
	or, in the general case,
	\begin{equation}
		\label{adxuns2compc}
		\begin{aligned}
	&\sum_{\alpha=0}^{n-1} \sum_{j=\alpha+1}^{n} \left[ (-1)^{j-1-\alpha}(a_{n-j}v)^{(j-1-\alpha)}(b)u^{(\alpha)}(b) - (-1)^{j-1-\alpha}(a_{n-j}v)^{(j-1-\alpha)}(a)u^{(\alpha)}(a) \right]\\
	& \quad +\begin{cases}
		\sum_{\alpha=0}^{n-1} u^{(\alpha)}(a) \int_a^b \gamma_1^\alpha(t)v(t)dt & \text{if } t_1=a \\
		0& \text{if } t_1 \neq a
	\end{cases}
	+\begin{cases}
		\sum_{\alpha=0}^{n-1} u^{(\alpha)}(b) \int_a^b \gamma_l^\alpha(t)v(t)dt & \text{if } t_l=b \\
		0& \text{if } t_l \neq b
	\end{cases}
	=0,
	\end{aligned}
	\end{equation}
	for all $u$ such that $V_i(u)=0$, $i=1, \ldots, n$.

	Here,
	\begin{equation*}
		\left\{
		\begin{aligned}
			L_n^{*}v(t)&= 0, \quad && t \in I, \\
			V_i^{*}(v) &= 0, \quad && i = 1, \ldots, n,
		\end{aligned}
		\right.
	\end{equation*}
	is the adjoint of Problem \eqref{partir},
	and $\Delta v^{(\alpha)}(t_k)$ denotes the jump of $v^{(\alpha)}(t)$ in $t=t_k$, that is
	\begin{equation*}
		\Delta v^{(\alpha)}(t_k)=v^{(\alpha)}(t_k^+)-v^{(\alpha)}(t_k^-),
	\end{equation*}
	where $v^{(\alpha)}(t_k^-)$, $v^{(\alpha)}(t_k^+)$ are the left and the right-hand limits of $v^{(\alpha)}$ in $t=t_k$.
\end{theorem}

\begin{proof}
	Let $u\in W^{n,1}(I)$ and $v\in PC_{\mathcal{T}}^{n-1}(I)$. Adding $\sum_{k=1}^{l}\sum_{\alpha=0}^{n-1}u^{(\alpha)}(t_k)\int_a^b\gamma_k^\alpha(s)v(s)\,\mathrm{d}s$ to both sides of the identity in Lemma~\ref{lemagreen} and grouping the terms with $t_k\in\{a,b\}$ with the boundary terms, we obtain
	\begin{equation}
		\label{greenKn}
		\begin{aligned}
			\int_a^b K_nu(t)\,v(t)\,\mathrm{d}t
			={}&
			\sum_{k=0}^{l}\int_{t_k}^{t_{k+1}}u(t)\,L_n^*v(t)\,\mathrm{d}t
			+\sum_{t_k\in(a,b)}\sum_{\alpha=0}^{n-1}
			\Bigl(\int_a^b\gamma_k^\alpha(s)v(s)\,\mathrm{d}s-\Lambda_k^\alpha(v)\Bigr)u^{(\alpha)}(t_k)\\
			&+\sum_{\alpha=0}^{n-1}
			\Bigl(\mathcal{B}_b^\alpha(v)+\varepsilon_b\int_a^b\gamma_l^\alpha(s)v(s)\,\mathrm{d}s\Bigr)u^{(\alpha)}(b)
			-\sum_{\alpha=0}^{n-1}
			\Bigl(\mathcal{B}_a^\alpha(v)-\varepsilon_a\int_a^b\gamma_1^\alpha(s)v(s)\,\mathrm{d}s\Bigr)u^{(\alpha)}(a),
		\end{aligned}
	\end{equation}
	where $\varepsilon_a=1$ if $t_1=a$ and $\varepsilon_a=0$ otherwise, and $\varepsilon_b=1$ if $t_l=b$ and $\varepsilon_b=0$ otherwise. Since $\Lambda_k^\alpha(v)$ is the left-hand side of \eqref{adxuns2compb}, condition \eqref{adxuns2compb} means that the second sum vanishes, and \eqref{adxuns2compc} means that the last two sums vanish for every $u$ with $V_i(u)=0$, $i=1,\ldots,n$.

	Let now $u\in D(K_n)$. If $v$ satisfies \eqref{adxuns2compb}--\eqref{adxuns2compc}, then $\langle K_nu,v\rangle_2=\langle u,K_n^*v\rangle_2$ with $K_n^*v=L_n^*v$ on $I\setminus\mathcal{T}$. Conversely, if this equality holds for every $u\in D(K_n)$, then the last three sums in \eqref{greenKn} vanish for every $u\in D(K_n)$. Given $t_p\in(a,b)$ and $\beta\in\{0,\ldots,n-1\}$, we choose $u\in D(K_n)$ vanishing near $a$, $b$ and $t_k$, $k\neq p$, with $u^{(\beta)}(t_p)=1$ and $u^{(\alpha)}(t_p)=0$ for $\alpha\neq\beta$. This gives $\Lambda_p^\beta(v)=\int_a^b\gamma_p^\beta(s)v(s)\,\mathrm{d}s$, that is, \eqref{adxuns2compb}. Hence the last two sums vanish for every $u\in D(K_n)$, which is \eqref{adxuns2compc}. Finally, if $t_1\neq a$ and $t_l\neq b$, then $\varepsilon_a=\varepsilon_b=0$ and \eqref{adxuns2compc} is equivalent to $V_i^*(v)=0$, $i=1,\ldots,n$, by the definition of the adjoint boundary conditions \cite[Section~1.4]{cabada2014greens}.
\end{proof}

\begin{remark}
	The preceding result connects two major types of equations that have been extensively studied in the literature: piecewise constant arguments equations and impulsive differential equations.
	
	Moreover, it also allows us to relate equations with piecewise constant arguments to equations with nonlocal boundary conditions.	
\end{remark}
\begin{remark}
For $\sigma \in L^1(I)$, the solution $u$ belongs to $W^{n,1}(I) \subset C^{n-1}(I)$, so $u^{(\alpha)}(t_k)$ is well defined, whereas the solution $v$ of the adjoint problem belongs to $PC_\mathcal{T}^{n-1}(I)$
and may have jumps at the points $t_k$. Accordingly, $H(\cdot,s)$ is more regular than $H(t, \cdot)$ \cite[Proposition~9]{cabada2025reflection}.
\end{remark}

An interesting particular case is $\gamma_k^{\alpha}=0$ for all $\alpha=1, \ldots, n-1$ and $k=1, \ldots, l$. That is, in the functional part of the equation we only consider $C_k^0(u)=u(t_k)$, for $k=1, \ldots, l$. In this case, we obtain the following result.
\begin{corollary}
	\label{adjuntocasosimple}
		The adjoint equation of 
	\begin{equation}
		\label{problemasimple}
		\left\{
		\begin{aligned}
			L_n u(t) +\sum_{k=1}^l  \gamma_k(t) u(t_k)&=0, \quad && t \in I, \, t_k \in I,\, k=1, \ldots, l, \\
			V_i(u) &= 0, \quad && i = 1, \ldots, n,
		\end{aligned}
		\right.
	\end{equation}
	is given by
	\begin{equation}
		\label{adjuntoproblemasimplea}
		L_n^* v(t)=0, \quad  t \in I \setminus \{t_1, \ldots, t_l\},
	\end{equation}
	with jump conditions for $k=1, \ldots, l$, $t_k \neq a$ and $t_k \neq b$
	\begin{equation}
		\label{adjuntoproblemasimpleb}
		\Delta v^{(r)}(t_k)=0,\quad r=0,\ldots,n-2,\qquad \Delta v^{(n-1)}(t_k)=(-1)^{n+1}\int_a^b{\gamma_k(t)v(t) \mathrm{d}t},
	\end{equation}
	and boundary conditions if $t_1 \neq a$ and $t_l \neq b$
	\begin{equation*}
		V_i^*(v)=0, \quad i=1, \ldots, n,
	\end{equation*}
	or, in general,
	\begin{equation}
		\label{adjuntoproblemasimplec}
		\begin{aligned}
			&\sum_{\alpha=0}^{n-1} \sum_{j=\alpha+1}^{n} \left[ (-1)^{j-1-\alpha}(a_{n-j}v)^{(j-1-\alpha)}(b)u^{(\alpha)}(b) - (-1)^{j-1-\alpha}(a_{n-j}v)^{(j-1-\alpha)}(a)u^{(\alpha)}(a) \right]\\
			& \quad +\begin{cases}
				u(a) \int_a^b \gamma_1(t)v(t)dt & \text{if } t_1=a \\
				0& \text{if } t_1 \neq a
			\end{cases}
			+\begin{cases}
				u(b) \int_a^b \gamma_l(t)v(t)dt & \text{if } t_l=b \\
				0& \text{if } t_l \neq b
			\end{cases}
			=0,
		\end{aligned}
	\end{equation}
	for all $u$ such that $V_i(u)=0$, $i=1, \ldots, n$.	
\end{corollary}
\begin{proof}
	Following Theorem \ref{adjpiece2}, it is easy to see that the adjoint problem satisfies:
	\begin{equation*}
		L_n^*v(t)=0, \quad t \in I \setminus \{t_1, \ldots, t_l\},
	\end{equation*}
	with jump conditions for $t_k$, $k=1, \ldots, l$, $t_k \neq a$ and $t_k \neq b$
	\begin{equation*}
		\sum_{j=1}^n \sum_{r=0}^{j-1} (-1)^{j-1} \binom{j-1}{r} a_{n-j}^{(j-1-r)} \Delta v^{(r)}(t_k)=\int_a^b{\gamma_k^0(t)v(t) \mathrm{d}t}.
	\end{equation*}
	However, since 
	\begin{equation*}
		\sum_{j=\alpha+1}^{n}\sum_{r=0}^{j-1-\alpha} (-1)^{j-1-\alpha} \binom{j-1-\alpha}{r} a_{n-j}^{(j-1-\alpha-r)} \Delta v^{(r)}(t_k)=0,
	\end{equation*}
	for all $\alpha=1, \ldots, n-1$, and this system is triangular (the equation corresponding to $\alpha$ involves $\Delta v^{(n-1-\alpha)}(t_k)$ with coefficient $(-1)^{n-1-\alpha}$, together with jumps of lower order), necessarily
	\begin{equation*}
		\Delta v^{(r)}(t_k)=0, \, \textup{ for all }r=0, \ldots, n-2.
	\end{equation*}
	Then the left-hand side of the jump condition for $\alpha=0$ reduces to $(-1)^{n-1}\Delta v^{(n-1)}(t_k)$, which gives \eqref{adjuntoproblemasimpleb} and completes the proof.
\end{proof}
\begin{remark}\label{validofuncional}
It is important to note that Theorems \ref{teoadxunto} and \ref{adjpiece2}
remain valid if we add to the operator $L_n$ given by \eqref{defl} a term of
the form $Bu(t)=m(t)\,u(\beta(t))$, where $m\in L^\infty(I)$ and
$\beta:I\rightarrow I$ is a $C^1$ diffeomorphism. Indeed, $B$ is a bounded
operator in $L^2(I)$ and, by means of the change of variables $s=\beta(t)$,
we have that
\begin{equation*}
 \langle Bu,v\rangle=\int_a^b m(t)\,u(\beta(t))\,v(t)\,dt
 =\int_a^b u(s)\,m(\beta^{-1}(s))\,v(\beta^{-1}(s))\,|(\beta^{-1})'(s)|\,ds,
\end{equation*}
so $B^*v(t)=m(\beta^{-1}(t))\,v(\beta^{-1}(t))\,|(\beta^{-1})'(t)|$. Since this
term does not produce boundary terms, the adjoint problem is obtained by adding
$B^*v$ to the adjoint equation, with the same jump and boundary conditions.
\end{remark}

The functionals $\mathcal{C}_k^\alpha(u)=u^{(\alpha)}(t_k)$ are not continuous in $L^2(I)$, so Theorem \ref{teoadxunto} does not apply to them. However, by Definition \ref{defdelta} they are represented by distributions, since $u^{(\alpha)}(t_k)= \langle (-1)^\alpha \delta^{(\alpha)}(t-t_k),u\rangle$. Replacing $\eta_k$ by these distributions in \eqref{oadjunto}, we are led to the problem
\begin{equation}
	K_n^*v(t):=L_n^* v(t) +\sum_{k=1}^l \sum_{\alpha=0}^{n-1} (-1)^{\alpha}\delta^{(\alpha)}(t-t_k)\int_a^b \gamma_k^\alpha(s)v(s) \mathrm{d}s =0, \qquad t \in I,
	\label{otraformadificil}
\end{equation}
which, unlike \eqref{adxuns2compa}, is meaningful at every $t \in I$, including the points $t_1, \ldots, t_l$. At an interior point $t_k$ the Dirac terms produce the jumps \eqref{adxuns2compb}. If $t_1=a$ or $t_l=b$, they act at the endpoints and modify the boundary conditions as in \eqref{adxuns2compc}. For this reason, no boundary conditions are written in \eqref{otraformadificil}: they are built into the weak formulation below, whose test functions need not vanish at $a$ and $b$, and Proposition~\ref{propotra} shows that they are \eqref{adxuns2compc}, which reduce to $V_i^*(v)=0$, $i=1,\ldots,n$, when $t_1\neq a$ and $t_l\neq b$.

\begin{definition}
\label{defotra}
	Let
	\[
	\mathcal{U}
	:=
	\left\{
	u\in C^n(I) :
	V_i(u)=0,\quad i=1,\ldots,n
	\right\}.
	\]
	
	A function
	\[
	v\in PC_{\mathcal{T}}^{n-1}(I)
	\]
	is said to be a weak adjoint solution of
	\eqref{otraformadificil} if
	\[
	\langle K_n^*v,u\rangle=0
	\qquad
	\text{for every }u\in\mathcal{U},
	\]
	where the duality pairing is understood on the closed interval $I$: the Dirac terms act as in Definition~\ref{defdelta}, and $L_n^*v$ acts by transposition, $\langle L_n^*v,u\rangle:=\int_a^b v(t)\,L_nu(t)\,\mathrm{d}t$ (which amounts to extending $v$ by zero outside $I$). In particular, $\langle K_n^*v,u\rangle=\int_a^b K_nu(t)\,v(t)\,\mathrm{d}t$.
	
	On the other hand, a function
	\[
	v\in PC_{\mathcal{T}}^{n-1}(I)
	\]
	is said to be a piecewise strong solution of
	\eqref{adxuns2compa}--\eqref{adxuns2compc} if
	$L_n^*v=0$ almost everywhere on each open subinterval determined
	by the points of $\mathcal{T}$ and $v$ satisfies the corresponding
	jump and boundary conditions.
\end{definition}
\begin{remark}
\label{rem518}
Testing with functions in $C_c^\infty((a,b))\subset\mathcal{U}$ recovers the equation on each subinterval and the jump conditions, whereas the functions in $\mathcal{U}$ that do not vanish at $a$ and $b$ retain the boundary terms in \eqref{adxuns2compc}. 
\end{remark}
This leads to the following result:
\begin{proposition}
\label{propotra}
A function $v \in PC_{\mathcal{T}}^{n-1}(I)$ is a weak adjoint solution of \eqref{otraformadificil} if and only if it is a piecewise strong solution of \eqref{adxuns2compa}--\eqref{adxuns2compc}.
\end{proposition}
\begin{proof}
	Let $v\in PC_{\mathcal{T}}^{n-1}(I)$ and $u\in\mathcal{U}$. By Definition~\ref{defotra}, $\langle K_n^*v,u\rangle=\int_a^bK_nu(t)\,v(t)\,\mathrm{d}t$, which is given by \eqref{greenKn}. If $v$ is a piecewise strong solution of \eqref{adxuns2compa}--\eqref{adxuns2compc}, the four sums in \eqref{greenKn} vanish, so $v$ is a weak adjoint solution of \eqref{otraformadificil}. Conversely, assume that $\langle K_n^*v,u\rangle=0$ for every $u\in\mathcal{U}$. Taking $u\in C_c^\infty((t_k,t_{k+1}))\subset\mathcal{U}$, only the first sum remains, so $L_n^*v=0$ a.e. on $(t_k,t_{k+1})$, $k=0,\ldots,l$, which is \eqref{adxuns2compa}. Then \eqref{adxuns2compb} and \eqref{adxuns2compc} follow as in the proof of Theorem~\ref{adjpiece2}, since the functions $u$ used there can be taken in $\mathcal{U}$.
\end{proof}

As a corollary of the previous result and taking into account Corollary \ref{adjuntocasosimple}, we also have that:
\begin{corollary}
	Problem \eqref{adjuntoproblemasimplea}--\eqref{adjuntoproblemasimplec} is equivalent, in the weak sense of Definition~\ref{defotra}, to:
\begin{equation}
	L_n^* v(t) +\sum_{k=1}^l \delta(t-t_k)\int_a^b \gamma_k(s)v(s) \mathrm{d}s =0, \qquad t \in I.
	\label{otraforma}
\end{equation}
\end{corollary}
This formulation allows us to obtain the Green's function of the impulsive problem following the idea of Theorem \ref{teoprincipal}. For $k=1, \ldots, l$ and $\alpha=0, \ldots, n-1$, we denote $J_k^\alpha(v)=\int_a^b \gamma_k^\alpha(s)v(s) \mathrm{d}s$. In Theorem \ref{obtergcomplicada} and Corollary \ref{teoprincipal22}, the functionals $J\colon L^\infty(I)\to\mathbb{R}$ are linear and continuous; as in Section \ref{explicit}, $J(\widehat{G}(\cdot,s))$ denotes the unique $\kappa\in L^\infty(I)$ such that $J\bigl(\int_a^b\widehat{G}(\cdot,s)\sigma(s)\,\mathrm{d}s\bigr)=\int_a^b\kappa(s)\sigma(s)\,\mathrm{d}s$ for all $\sigma\in L^1(I)$, and analogously for $J(G(\cdot,s))$.
\begin{theorem}
	\label{obtergcomplicada}
	Assume that the impulse-free version of Problem 
	 \eqref{adxuns2compa}--\eqref{adxuns2compc} has $\widehat{G}$ as its unique Green's function, and let $\sigma \in L^1(I)$, $J_k^\alpha: L^\infty(I) \rightarrow \mathbb{R}$, $k=1, \ldots, l$, $\alpha=0, \ldots, n-1$ be such that 
	\begin{equation}
		\det(\widehat{A}+I_{ln}) \neq 0
	\end{equation}
	with $I_{ln}$ the identity matrix of order $ln$ and $\widehat{A}=(\widehat{a}_{ij})_{ln \times ln}$ given by
	\begin{equation}
	\label{ajemplo}
		\widehat{a}_{ij}=J_m^\beta \left(\frac{\partial^{\alpha} \widehat{G}}{\partial s^{\alpha}}(\cdot, t_k) \right)
	\end{equation}
	where $\alpha=(j-1) \mod n$, $k= \lceil \frac{j}{n}\rceil$, $\beta=(i-1) \mod n$ and $m= \lceil \frac{i}{n}\rceil$.
	Then Problem \eqref{otraformadificil}, with $\sigma$ on the right-hand side, admits a unique solution $u \in PC_{\mathcal{T}}^{n-1}(I)$ and its Green's function $\widehat{H}$ is given by
	\begin{equation}
	\label{528}
		\widehat{H}(t,s)=\widehat{G}(t,s)-\sum_{k=1}^{l}\sum_{\alpha=0}^{n-1}\frac{\partial^{\alpha} \widehat{G}}{\partial s^{\alpha}}(t,t_k)\sum_{i=1}^{ln}\widehat{\tilde{a}}_{j,i}\tilde{J}_i(\widehat{G}(\cdot,s))
	\end{equation}
	where $\widehat{\tilde{A}}=(\widehat{\tilde{a_{ij}}})_{ln \times ln}$ denotes the inverse matrix of $\widehat{A}+I_{ln}$, $j=(k-1)n+\alpha+1$ and
	\begin{equation}
		\label{tildej}
		\tilde{J}_i(\widehat{G}(\cdot,s))=J_m^\beta(\widehat{G}(\cdot,s)).
	\end{equation}
	with $\beta=(i-1) \mod n$ and $m= \lceil \frac{i}{n} \rceil$.
\end{theorem}
\begin{proof}
	Since $\widehat{G}$ is the unique Green's function of the impulse-free Problem \eqref{adxuns2compa}--\eqref{adxuns2compc} and, by the definition of $\delta^{(\alpha)}$, $\int_a^b \widehat{G}(t,s)(-1)^\alpha\delta^{(\alpha)}(s-t_k)\,\mathrm{d}s=\frac{\partial^{\alpha}\widehat{G}}{\partial s^{\alpha}}(t,t_k)$, the solution satisfies
	\begin{equation}
		u(t)=\int_a^b \widehat{G}(t,s)\sigma(s)\,\mathrm{d}s-\sum_{k=1}^{l}\sum_{\alpha=0}^{n-1}\frac{\partial^{\alpha}\widehat{G}}{\partial s^{\alpha}}(t,t_k)\,J_k^\alpha(u).
		\label{inicialdificil}
	\end{equation}
	Applying the operators $J_m^\beta$ to both sides, we obtain the linear system $(\widehat{A}+I_{ln})X=B$, where $\widehat{A}$ is given by \eqref{ajemplo} and $X,B\in\mathbb{R}^{ln}$ have components
	\begin{equation*}
		x_j=J_k^\alpha(u), \qquad b_i=\int_a^b \tilde{J}_i(\widehat{G}(\cdot,s))\,\sigma(s)\,\mathrm{d}s,
	\end{equation*}
	with $j=(k-1)n+\alpha+1$ and $\tilde{J}_i$ given by \eqref{tildej}. Hence $J_k^\alpha(u)=\sum_{i=1}^{ln}\widehat{\tilde{a}}_{ji}\,b_i$, and substituting into \eqref{inicialdificil} we obtain the expression of $\widehat{H}$ given in the statement.
\end{proof}

As a particular case, consider the impulsive problem
\begin{equation}
\label{proprin2}
\left\{
\begin{aligned}
& L_n u(t)=\sigma(t), && t\in I\setminus\{t_1,\dots,t_l\},\\
& \Delta u^{(r)}(t_k)=0,\ \ r=0,\dots,n-2,\qquad
  \Delta u^{(n-1)}(t_k)=-J_k(u), && k=1,\dots,l,\\
& V_i(u)=0, && i=1,\dots,n,
\end{aligned}
\right.
\end{equation}
where $L_n$ and $V_i$ are defined in \eqref{defl} and \eqref{defv}, respectively,
$a<t_1<\dots<t_l<b$, $\sigma\in L^1(I)$ and
$J_k\colon L^\infty(I)\to\mathbb{R}$, $k=1,\dots,l$, are linear and continuous functionals.
\begin{corollary}
	\label{teoprincipal22}
	Assume that Problem \eqref{partir} has $G$ as its unique Green's function and let $\sigma \in L^{1}(I)$, $J_{k}: L^\infty(I) \rightarrow \mathbb{R}$, $k=1, \ldots, l$ be linear and continuous functionals such that
	\begin{equation}
		\det{(\overline{A}+I_{l})} \neq 0
		\label{condiciondet2}
	\end{equation}
	with $I_{l}$, the identity matrix of order $l$ and $\overline{A}=(\overline{a}_{ij})_{l \times l}$ given by
	\begin{equation}
		\overline{a}_{i,j} = J_i(G(\cdot,t_j)), \quad i, j \in \{1, \ldots, l\}.
		\label{defa2}
	\end{equation}
	Then, Problem \eqref{proprin2} admits a unique solution $u \in PC_{\mathcal{T}}^{n-1}(I)$, and its Green's function $\overline{H}$ is given by
	\begin{equation}
		\overline{H}(t,s)=G(t,s)-\sum_{i=1}^{l}G(t,t_i) \sum_{j=1}^{l} \overline{\tilde{a}}_{i,j}J_j(G(\cdot,s)),
		\label{ecprincipal2}
	\end{equation}
	with $\overline{\tilde{A}}=(\overline{\tilde{a}}_{ij})_{l \times l}$ denoting the inverse of matrix $\overline{A}+I_l$.
\end{corollary}

Alternatively, since \eqref{adxuns2compa}--\eqref{adxuns2compc} is the adjoint of \eqref{adxunscomp}, its Green's function can also be obtained from Theorem \ref{teoprincipal} and Corollary \ref{coroadjoint}. Since $\widehat{G}(t,s)=G(s,t)$, we have $\widehat{A}=A^T$, $\frac{\partial^\alpha}{\partial s^\alpha}\widehat{G}(t,t_k)=\frac{\partial^\alpha}{\partial t^\alpha}G(t_k,t)$ and $J_m^\beta (\widehat{G}(\cdot,s))=\int_a^b G(s,r) \gamma_m^\beta(r) \mathrm{d}r$, and substituting into \eqref{528} yields \eqref{ecprincipal} with $t$ and $s$ interchanged, that is,
\begin{equation}
\label{igualdades}
	\widehat{H}(t,s)=H(s,t)=H^*(t,s).
\end{equation}

As a consequence, every sign property of the Green's function of an equation with point evaluation is inherited by the corresponding impulsive problem.
\begin{corollary}
\label{cortrans}
Assume that Problem \eqref{partir} has a unique Green's function $G$, let $\mathcal{C}_k(u)=u(t_k)$, $k=1, \ldots, l$, and assume that \eqref{condiciondet} holds. Let $H$ be the Green's function of Problem \eqref{problemasimple}. Then, for every $\sigma \in L^1(I)$, the impulsive problem 
\begin{equation*}
L_n^* v(t)=\sigma(t), \quad t \in I \setminus\{t_1, \ldots, t_l\},
\end{equation*}
coupled with the jump conditions \eqref{adjuntoproblemasimpleb}, including the continuity of $v,\ldots,v^{(n-2)}$ at the points $t_k$, and the boundary conditions \eqref{adjuntoproblemasimplec}, has a unique solution, given by
\begin{equation*}
v(t)=\int_a^b H(s,t) \sigma(s) \mathrm{d}s.
\end{equation*}
In particular, if $H>0$ (respectively, $H<0$) on $\mathring{I} \times \mathring{I}$, then $\sigma \geq 0$ a.e. on $I$ with $\sigma \neq 0$, implies $v>0$ (respectively, $v<0$) on $\mathring{I} \setminus \{t_1, \ldots, t_l\}$.
\end{corollary}

\section{Examples and Applications}
\label{illustrative}
In this section, we present a series of examples illustrating the
usefulness of the results developed above. In particular, we show how
the proposed theoretical framework can be used to address several
problems of interest in the literature, either by deriving their
Green's functions directly or by identifying them as the adjoint
problems of other problems that have already been studied.

\subsection{Impulsive Differential Equation}

As we have mentioned throughout the previous section, Theorem \ref{adjpiece2} relates the equations with piecewise constant arguments to the differential impulsive equations, providing a direct method for computing their Green's functions and characterizing their constant-sign region. Next, we present an example.

\subsubsection{Constant Sign for a First-Order Impulsive Problem with Reflection}
\label{ejemploreflection}

In \cite{cabada2025reflection}, we consider the following problem:
\begin{equation}
	\begin{aligned}
		u'(t)+mu(-t)+Mu([t])&= \sigma(t), \quad t \in J:=[-T,T], \\
		u(-T)&=u(T).
	\end{aligned}
	\label{artlargo}
\end{equation}
Here $m$ and $M$ are real constants, not both zero, $T>0$ and the function $[t]$ is given by \eqref{parteenteira}. Notice that $[t]=0$ for all $t \in (-1,1)$.

Arguing as in Corollary \ref{adjuntocasosimple} and taking into account Remark \ref{validofuncional}, one can see that the adjoint operator, for non-integer $T$, is given by
\begin{equation}
	\label{artlargoad}
	\begin{aligned}
		-v'(t)+mv(-t)&= \sigma(t), \quad t \in J, \\
		\Delta v(i)&=M\int_{J_i}v(t)\,\mathrm{d}t, \quad i=[-T],\ldots,[T],\\
		v(-T)&=v(T),
	\end{aligned}
\end{equation}
where $J_i$ are the intervals defined in Example \ref{ej36}.

In \cite{cabada2025reflection}, we obtain and characterize the Green's function of Problem \eqref{artlargo}. Therefore, taking into account \eqref{igualdades}, we can deduce many properties of the Green's function of Problem \eqref{artlargoad}. In particular, we have the following lemma:
\begin{lemma}
\label{dire}
	Consider the Green's functions of problems \eqref{artlargo} and \eqref{artlargoad} and $T<1$. Then
	\begin{itemize}
		\item If $m \in \left(-\frac{\pi}{4T}, \frac{\pi}{4T}\right)$, $m \neq 0$ and $M \in \left(-m,\frac{1}{2}m(-1+\cot{(mT)}) \right)$ or $m=0$ and $M \in \left(0,\frac{1}{2T} \right)$, they are strictly positive on $J \times J$.
		\item If $m \in \left(-\frac{\pi}{4T}, \frac{\pi}{4T}\right)$, $m \neq 0$ and $M \in \left(-\frac{1}{2}m(1+\cot{(mT)}),-m \right)$ or $m=0$ and $M \in \left(-\frac{1}{2T},0 \right)$, they are strictly negative on $J \times J$.
	\end{itemize}	
\end{lemma}
\begin{remark*}
In the first item of Lemma \ref{dire}, with $m\neq0$, the upper bound for $M$ is optimal. Indeed, by Corollary \ref{casosimple}, the Green's function of \eqref{artlargo} is $H(t,s)=G(t,s)-\frac{M}{m+M}\,G(0,s)$, where $G$ is the Green's function of \eqref{artlargo} with $M=0$; if $t<s<0$ and $(t,s)\to(0^-,0^-)$, then $H(t,s)\to\frac{m\,(1+\cot(mT))}{2(m+M)}-1$, which vanishes exactly when $M=\frac12\,m\,(-1+\cot(mT))$.
\end{remark*}
Therefore, we have now obtained sufficient conditions for the constant sign of the Green's function for Problem \eqref{artlargoad}, and we can apply numerous fixed point theory results to study the nonlinear case.

\begin{corollary}
\label{cor62}
Let $0<T<1$, and consider the impulsive problem
\begin{equation}
\label{eqimre}
\begin{aligned}
w'(t)+m\,w(-t)&=\sigma(t), \quad t \in J \setminus \{0\}, \\
\Delta w(0)&=-M \int_{-T}^T w(s) \mathrm{d}s, \quad w(-T)=w(T),
\end{aligned}
\end{equation}
where $\Delta w(0)=w(0^+)-w(0^-)$ and $\sigma \in L^1(J)$.
\begin{enumerate}
\item If $m\in\bigl(-\frac{\pi}{4T},\frac{\pi}{4T}\bigr)$, $m\neq0$ and
$M\in\bigl(-m,\frac12 m(-1+\cot(mT))\bigr)$, or $m=0$ and
$M\in\bigl(0,\frac{1}{2T}\bigr)$, then Problem \eqref{eqimre}
has a unique solution and $\sigma\ge0$ a.e.\ on $J$ with $\sigma\not\equiv0$
implies $w>0$ on $J\setminus\{0\}$.
\item If $m\in\bigl(-\frac{\pi}{4T},\frac{\pi}{4T}\bigr)$, $m\neq0$ and
$M\in\bigl(-\frac12 m(1+\cot(mT)),-m\bigr)$, or $m=0$ and
$M\in\bigl(-\frac{1}{2T},0\bigr)$, then Problem \eqref{eqimre}
has a unique solution and $\sigma\ge0$ a.e.\ on $J$ with $\sigma\not\equiv0$
implies $w<0$ on $J\setminus\{0\}$.
\end{enumerate}
\end{corollary}
\begin{proof}
Let $w$ be a solution of \eqref{eqimre} and set $v(t)=w(-t)$. A direct computation shows that $v$ solves \eqref{artlargoad} with right-hand side $\sigma(-t)$, and conversely. Thus, by Corollary \ref{cortrans},
\begin{equation*}
w(t)=v(-t)=\int_{-T}^T H(s,-t) \sigma(-s) \mathrm{d}s=\int_{-T}^TH(-r,-t)\sigma(r) \mathrm{d}r.
\end{equation*}
Therefore, the Green's function of \eqref{eqimre} is
\begin{equation*}
K(t,r)=H(-r,-t).
\end{equation*}
Since $K$ has the same sign as $H$, the conclusion follows from Lemma \ref{dire}.
\end{proof}

\subsubsection{Constant Sign for Higher-Order Impulsive Problems}
We now show how the preceding relation can be used to characterize the parameter values for which the Green's function of an impulsive problem has a constant sign. For this purpose, we use a recent result by Labovskiy \cite{Lab}, which characterizes the sign of the Green's function for two-point problems with functional dependence in terms of the first eigenvalue of certain associated problems. Sign-constancy of Green's functions of impulsive problems has also been studied, for instance, in \cite{domoshnitsky2014sign,DomoshnitskyMizgireva2019a,DomoshnitskyMizgireva2019b}.

Let $n\geq 2$. For each $q\in\{0,\dots,n\}$, we consider the boundary conditions
\begin{equation}\label{eq:Bq}
 u^{(i)}(a)=0,\ i=0,\dots,n-q-1,\qquad u^{(i)}(b)=0,\ i=0,\dots,q-1,
 \tag{$B_q$}
\end{equation}
that is, $n-q$ conditions at $t=a$ and $q$ conditions at $t=b$. Given
$p\in\{1,\dots,n-1\}$, in \cite{Lab} the following problem is studied:
\begin{equation}\label{eq:lab}
 (-1)^{p+1}u^{(n)}(t)-\lambda\,Qu(t)=f(t),\quad t\in I,
 \qquad\text{coupled with \eqref{eq:Bq} for } q=p,
\end{equation}
where $\lambda\in\mathbb{R}$, $f\in L^1(I)$ and
\begin{equation*}
 Qu(t)=\int_a^b u(s)\,d_sr(t,s),\quad t\in I.
\end{equation*}
Here $r(\cdot,s)$ is measurable for all $s\in I$, $r(t,\cdot)$ is
nondecreasing for a.e. $t\in I$ and continuous at $s=a$ and $s=b$,
$r(t,a)=0$, $r(\cdot,b)\in L^1(I)$ and $r(\cdot,b)\not\equiv 0$.

Moreover, for each $q\in\{0,\dots,n\}$, we denote by $\lambda_q$ the smallest
positive value of $\lambda$ for which the problem
\begin{equation}\label{eq:eig}
 (-1)^{q}u^{(n)}(t)-\lambda\,Qu(t)=0,\quad t\in I,
 \qquad\text{coupled with \eqref{eq:Bq}},
\end{equation}
has a nontrivial solution. If there is no such value, we take
$\lambda_q=+\infty$. Finally, let $G_q$ be the Green's function of
$u^{(n)}(t)=\sigma(t)$, $t\in I$, coupled with \eqref{eq:Bq}, which satisfies
$(-1)^q\,G_q\ge 0$ on $I\times I$.  

\begin{theorem}[{\cite[Theorems 3.1, 4.1 and 4.2, Lemma 3.2]{Lab}}]
\label{thm:lab}
Let $G_\lambda$ be the Green's function related to Problem \eqref{eq:lab}. Then, the
following properties are fulfilled:
\begin{enumerate}
 \item If $-\lambda_p<\lambda<\min\{\lambda_{p-1},\lambda_{p+1}\}$, then
 $G_\lambda$ exists and $G_\lambda(t,s)\le 0$ for all $(t,s)\in I\times I$.
 \item If $2\le p\le n-2$, then $G_\lambda$ exists and
 $G_\lambda(t,s)\le 0$ for all $(t,s)\in I\times I$, if and only if,
 $-\lambda_p<\lambda\le\min\{\lambda_{p-1},\lambda_{p+1}\}$.
 \item For each $q\in\{0,\dots,n\}$, $\lambda_q$ is the inverse of the
 spectral radius of the operator
 \begin{equation*}
  K_qz(t)=Q\left(\int_a^b(-1)^qG_q(\cdot,s)\,z(s)\,ds\right)(t),
  \quad t\in I.
 \end{equation*}
\end{enumerate}
\end{theorem}
Although in \cite{Lab} the interval $[0,l]$ is considered, the same result
holds on $I$ by means of a translation.

Now, let $p\in\{1,\dots,n-1\}$ and consider the problem
\begin{equation}\label{eq:point}
 u^{(n)}(t)+M\sum_{k=1}^{l}\gamma_k(t)\,u(t_k)=\sigma(t),\quad t\in I,
 \qquad\text{coupled with \eqref{eq:Bq} for } q=p,
\end{equation}
where $M\in\mathbb{R}$, $a<t_1<\dots<t_l<b$, $\sigma\in L^1(I)$ and
$\gamma_k\in L^1(I)$ is such that $\gamma_k\ge 0$ a.e. on $I$ and
$\gamma_k\not\equiv 0$ for all $k=1,\dots,l$. From Corollary \ref{adjuntocasosimple}, the adjoint of Problem \eqref{eq:point} is the impulsive problem
\begin{equation}\label{eq:imp}
	\left\{
	\begin{aligned}
		 (-1)^n v^{(n)}(t)&=\sigma(t), && t\in I\setminus\{t_1,\dots,t_l\},\\
		 \Delta v^{(r)}(t_k)&=0,\ \ r=0,\dots,n-2,\qquad
		&&\Delta v^{(n-1)}(t_k)=(-1)^{n+1}M\int_a^b\gamma_k(s)\,v(s)\,ds,  k=1,\dots,l,\\
		 v^{(i)}(a)&=0,\ \ i=0,\dots,p-1,\qquad &&v^{(i)}(b)=0,\ \ i=0,\dots,n-p-1.
	\end{aligned}
	\right.
\end{equation}
We denote by $H_M$ and $\widehat H_M$ the Green's functions of problems
\eqref{eq:point} and \eqref{eq:imp}, respectively. Following \eqref{defa}, we define
the matrix $A_q=(a^q_{ij})_{l\times l}$ given by
\begin{equation*}
 a^q_{ij}=(-1)^q\int_a^b G_q(t_i,s)\,\gamma_j(s)\,ds,
 \qquad i,j\in\{1,\dots,l\},
\end{equation*}
and we denote by $\rho(A_q)$ its spectral radius.

As a consequence of Theorem \ref{thm:lab}, we arrive at the following result.

\begin{theorem}\label{thm:imp}
Let $n\ge 2$ and $p\in\{1,\dots,n-1\}$, and let $\lambda_q$ be the numbers
defined in \eqref{eq:eig} for $Qu(t)=\sum_{k=1}^{l}\gamma_k(t)\,u(t_k)$. Then
\begin{equation*}
 \lambda_q=\frac{1}{\rho(A_q)},\qquad q\in\{0,\dots,n\},
\end{equation*}
where $\lambda_q=+\infty$ if $\rho(A_q)=0$. Moreover, the Green's function
$\widehat H_M$ of the impulsive Problem \eqref{eq:imp} satisfies the following
properties:
\begin{enumerate}
 \item If $-\lambda_p<(-1)^pM<\min\{\lambda_{p-1},\lambda_{p+1}\}$, then
 $(-1)^p\,\widehat H_M(t,s)\ge 0$ for all $(t,s)\in I\times I$.
 \item If $2\le p\le n-2$, then $(-1)^p\,\widehat H_M(t,s)\ge 0$ for all
 $(t,s)\in I\times I$, if and only if,
 \begin{equation*}
  -\lambda_p<(-1)^pM\le\min\{\lambda_{p-1},\lambda_{p+1}\}.
 \end{equation*}
\end{enumerate}
The same properties hold for the Green's function $H_M$ of Problem
\eqref{eq:point}.
\end{theorem}

\begin{proof}
From \eqref{igualdades}, we have that $\widehat H_M(t,s)=H_M(s,t)$, so it is enough to
study the sign of $H_M$.

Multiplying the equation in \eqref{eq:point} by $(-1)^{p+1}$, we obtain
Problem \eqref{eq:lab} with $\lambda=(-1)^pM$, $f=(-1)^{p+1}\sigma$ and
\begin{equation*}
 Qu(t)=\sum_{k=1}^{l}\gamma_k(t)\,u(t_k)=\int_a^b u(s)\,d_sr(t,s),
 \qquad\text{where } r(t,s)=\sum_{k=1}^{l}\gamma_k(t)\,\chi_{[t_k,b]}(s).
\end{equation*}
This function $r$ satisfies the assumptions of Theorem \ref{thm:lab}. Indeed,
$r(t,\cdot)$ is nondecreasing, since $\gamma_k\ge 0$, it is continuous at $s=a$
and $s=b$ and $r(t,a)=0$, since $t_k\in(a,b)$, and $r(\cdot,b)=\sum_{k=1}^{l}\gamma_k\in L^1(I)$ is not identically zero.
Moreover, $G_\lambda=(-1)^{p+1}H_M$, so $(-1)^pH_M\ge 0$ if and only if
$G_\lambda\le 0$. Therefore, Assertions 1 and 2 follow from Theorem
\ref{thm:lab}.

It remains to prove that $\lambda_q=1/\rho(A_q)$. In this case, the operator
of Theorem \ref{thm:lab} is
\begin{equation*}
 K_qz(t)=\sum_{k=1}^{l}\gamma_k(t)\int_a^b(-1)^qG_q(t_k,s)\,z(s)\,ds .
\end{equation*}
If $K_qz=\mu z$ with $\mu\neq 0$ and $z\not\equiv 0$, then
$z=\sum_{k=1}^{l}c_k\gamma_k$, evaluating the integrals, we obtain that
$A_q\,c=\mu\,c$, with $c=(c_1,\dots,c_l)^T\neq 0$. Conversely, if
$A_q\,c=\mu\,c$ with $\mu\neq0$ and $c\neq 0$, then
$z=\sum_{k=1}^{l}c_k\gamma_k$ satisfies $K_qz=\mu z$ and $z\not\equiv 0$.
Hence, the spectral radius of $K_q$ is $\rho(A_q)$, which completes the proof.
\end{proof}
\begin{remark}\label{rem:imp}
It is important to note that the numbers $\lambda_q$ are obtained from the
spectral radius of a matrix of order $l$. In particular, if $l=1$, then
\begin{equation*}
 \lambda_q=\frac{1}{(-1)^q\,w_q(t_1)},
\end{equation*}
where $w_q(t)=\int_a^bG_q(t,s)\,\gamma_1(s)\,ds$ is the solution of
$u^{(n)}=\gamma_1$ coupled with \eqref{eq:Bq}.

On the other hand, the cases $p=1$ and $p=n-1$ are not included in
Assertion 2 because in \cite{Lab} only the sufficient condition is proved for
them.
\end{remark}

\subsubsection{Constant Sign for a Fourth-Order Impulsive Problem}
In the previous example, the sign of the Green's function was known for
the equation with piecewise constant arguments and it was transferred to the
impulsive one. Next, we present an example in which the constant sign region
is obtained from Theorem \ref{thm:imp}.

Consider the fourth-order problem
\begin{equation}\label{eq:beam-epca}
\left\{
\begin{array}{l}
 u^{(4)}(t)+M\,u([t])=\sigma(t),\quad t\in J,\\[2mm]
 u(-T)=u'(-T)=u(T)=u'(T)=0,
\end{array}
\right.
\end{equation}
where $M\in\mathbb{R}$, $T\in(0,1)$ and $\sigma\in L^1(J)$. Notice that
$[t]=0$ for all $t\in(-1,1)$, so \eqref{eq:beam-epca} is a problem of the form
\eqref{eq:point} with $n=4$, $p=2$, $l=1$, $t_1=0$ and $\gamma_1\equiv 1$.
Then, its adjoint is the impulsive problem
\begin{equation}\label{eq:beam-imp}
\left\{
\begin{array}{l}
 v^{(4)}(t)=\sigma(t),\quad t\in J\setminus\{0\},\\[2mm]
 \Delta v(0)=\Delta v'(0)=\Delta v''(0)=0,\qquad
 \Delta v'''(0)=-M\displaystyle\int_{-T}^{T}v(s)\,ds,\\[3mm]
 v(-T)=v'(-T)=v(T)=v'(T)=0,
\end{array}
\right.
\end{equation}
which models a beam clamped at both ends with a point force at the midpoint
proportional to its total deflection.

From Remark \ref{rem:imp}, we have that $\lambda_q=1/\bigl((-1)^qw_q(0)\bigr)$,
$q=1,2,3$, where $w_q$ is the solution of $u^{(4)}=1$ on $J$ with $4-q$
boundary conditions at $t=-T$ and $q$ at $t=T$. It is not difficult to verify
that
\begin{equation*}
 w_1(t)=\frac{(t+T)^3(t-T)}{24},\qquad
 w_2(t)=\frac{(T^2-t^2)^2}{24},\qquad
 w_3(t)=-\frac{(T-t)^3(t+T)}{24},
\end{equation*}
from which we deduce that
\begin{equation*}
 \lambda_1=\lambda_2=\lambda_3=\frac{24}{T^4}.
\end{equation*}
As a consequence of Theorem \ref{thm:imp}, we obtain the following result.

\begin{corollary}\label{corlab}
Let $T\in(0,1)$. Then the Green's functions of problems \eqref{eq:beam-epca}
and \eqref{eq:beam-imp} are nonnegative on $J\times J$, if and only if,
\begin{equation*}
 M\in\left(-\frac{24}{T^4},\,\frac{24}{T^4}\right].
\end{equation*}
\end{corollary}

If $T>1$ is not an integer, the impulses are located at the integers of
$(-T,T)$ and the same procedure can be applied. In this case, $\lambda_1$,
$\lambda_2$ and $\lambda_3$ are obtained from matrices of order $2[T]+1$.

\subsection{Differential Equations With Nonlocal Boundary Conditions}

Following Theorem \ref{adjpiece2}, it is easy to see that, when computing the adjoint of differential equations with piecewise constant arguments, we can obtain, in addition to impulsive differential equations, ordinary differential equations with nonlocal boundary conditions.

This is useful because it provides an alternative method to that of \cite{cabada2021green} for computing the Green's function for equations with nonlocal boundary conditions, simply by taking into account that $H^{*}(t,s)=H(s,t)$, and it also allows us to deduce properties of the adjoint problem from the original one.

As an example, we shall study the following problem, which has already been analyzed in \cite{cabada2021green}.

\subsubsection{Constant Sign for a First-Order Problem with  Nonlocal Boundary Conditions}
Consider the problem:
\begin{equation}
	\left\{
	\begin{aligned}
		u'(t)+Mu(t) &= \sigma(t), \quad && t \in \hat{I}:=[0, 1], \\
		u(0)-u(1) &= m \int_0^1{u(s)\mathrm{d}s}.
	\end{aligned}
	\right.
	\label{nolocal}
\end{equation}

Following a procedure similar to that of Corollary \ref{adjuntocasosimple}, it is straightforward to see that the adjoint problem is given by:
\begin{equation}
	\left\{
	\begin{aligned}
		-v'(t)+Mv(t)-mv(0) &= \sigma(t), \quad && t \in \hat{I}, \\
		v(0)-v(1) &= 0.
	\end{aligned}
	\right.
	\label{nolocalad}
\end{equation}
Theorem \ref{teoprincipal} allows us to determine the Green's function of Problem \eqref{nolocalad} and, by applying \eqref{igualdades}, to recover the corresponding Green's function of Problem \eqref{nolocal}. It is straightforward to verify that this coincides with the Green's function obtained in \cite[Equation (24)]{cabada2021green}.

Consequently, once the Green's function of \eqref{nolocal} is characterized, that of Problem \eqref{nolocalad} is also determined. In particular, we may state the following result.
\begin{lemma}
	Consider the Green's functions of problems \eqref{nolocal} and \eqref{nolocalad}. Then
	\begin{itemize}
		\item If $M \neq 0$ and $m \in \left(\frac{M}{1-e^{M}},M \right)$, they are strictly positive on $\hat{I} \times \hat{I}$.
		\item If $M \neq 0$ and $m \in \left(M,\frac{Me^M}{e^M-1} \right)$, they are strictly negative on $\hat{I} \times \hat{I}$.
	\end{itemize}	
\end{lemma}
\begin{proof}
A direct computation (or Theorem \ref{teoprincipal} applied to \eqref{nolocalad} together with \eqref{igualdades}) shows that, for $M\neq0$ and $m\neq M$, the Green's function of \eqref{nolocal} is
\[
H(t,s)=\varphi(t)\Bigl(R(t,s)+\frac{m}{M-m}\Bigr),\qquad \varphi(t)=\frac{e^{M(1-t)}}{e^{M}-1},
\]
where $R(t,s)=e^{Ms}$ if $s<t$ and $R(t,s)=e^{M(s-1)}$ if $s>t$. Since $\varphi$ has the sign of $M$ and $e^{-|M|}<R<e^{|M|}$, with both bounds approached, the result follows by elementary computations. By \eqref{igualdades}, the Green's function of \eqref{nolocalad} is $H(s,t)$. For $M=0$, the Green's function of \eqref{nolocal} is $H(t,s)=-\frac{1}{m}-(1-s)+\chi_{\{s<t\}}$, so that $H>0$ if $m\in(-1,0)$ and $H<0$ if $m\in(0,1)$, which are the limits of the intervals in the statement as $M\to0$.
\end{proof}

\subsection{First Order Integral Periodic Problem}
\label{sec63}
This part is devoted to illustrating the applicability of
Theorem~\ref{teoprincipal} and some of the results established in the previous sections. To this end, we analyze a simple problem whose constant-sign region is completely characterized in Theorem \ref{lema66}.

Consider the following problem
\begin{equation}
	\left\{
	\begin{aligned}
		u'(t)+m\,u(t)+M\,\int_{0}^{1}{u(s) \mathrm{d}s} &= \sigma(t), \quad && t \in \hat{I} := [0, 1], \\
		u(0)-u(1)&=0,
	\end{aligned}
	\right.
	\label{firstorder}
\end{equation}
with $m$ and $M \in \mathbb{R}$ and $\sigma \in L^{1}(\hat{I})$. 

This is a simple and particular case of a Fredholm integro-differential equation with constant kernel $K(t,s)=1$. In the following, we analyze how this problem can be studied using the framework introduced above.

For $M=0$, Problem \eqref{firstorder} reduces to \eqref{contra} with $M=0$, whose Green's function $G_m$ is given by \eqref{defg} for $m \neq 0$.

Then, using Theorem \ref{teoprincipal}, or directly Corollary \ref{casosimple}, we see that when $m \neq 0$ and $m+M \neq 0$, the Green's function of Problem \eqref{firstorder} is given by:
\begin{equation}
	H_{m,M}(t,s) = G_{m}(t,s) -  \int_{0}^{1} G_{m}(l,s)\, \mathrm{d}l  
	\frac{M \int_{0}^{1} G_{m}(t,r)\, \mathrm{d}r}{1 + M \int_{0}^{1} \left( \int_{0}^{1} G_{m}(l,r)\, \mathrm{d}l \right) \mathrm{d}r}.
	\label{hfirst}
\end{equation}
It is easy to verify that
\begin{equation*}
	\int_0^1{G_m(t,r) \mathrm{d}t}=\int_{0}^{1}{G_{m}(t,r) \mathrm{d}r}=\frac{1}{m} \textup{ for all }t,r \in \hat{I} \textup{ and }m \neq 0,
\end{equation*}
from which we deduce that
\begin{equation}
	H_{m,M}(t,s) = G_{m}(t,s) -  \frac{M}{m(m+M)}.  
	\label{hfirstsim}
\end{equation}
On the other hand, we can deduce the following symmetry property:
\begin{lemma}
\label{sim2}
	Assume that Problem \eqref{firstorder} has a unique solution and let $H_{m,M}$ be its related Green's function. Then, the following symmetry property holds:
	\begin{equation}
		H_{m,M}(t,s)=-H_{-m,-M}(1-t,1-s) \textup{ for all } (t,s) \in \hat{I} \times \hat{I}.
		\label{simetria}
	\end{equation}
\end{lemma}
\begin{proof}
	Let $u$ be the unique solution of Problem \eqref{firstorder} given by
	\begin{equation*}
		u(t)=\int_{0}^{1}{H_{m,M}(t,s) \sigma(s) \mathrm{d}s}.
	\end{equation*}
	It is straightforward to verify that $v(t):=u(1-t)$ is the unique solution of problem
	\begin{equation*}
		\left\{
		\begin{aligned}
			v'(t)-m\,v(t)-M\int_{0}^{1}{v(s) \mathrm{d}s} &= -\sigma(1-t), \quad && t \in \hat{I}, \\
			v(0)-v(1)&=0,
		\end{aligned}
		\right.
	\end{equation*}
	thus
	\begin{equation*}
		v(t)=-\int_0^1 H_{-m,-M}(t,s) \sigma(1-s) \mathrm{d}s.
	\end{equation*}
	Moreover,
	\begin{equation*}
		u(1-t)=\int_0^1 H_{m,M}(1-t,s) \sigma(s) \mathrm{d}s=\int_0^1 H_{m,M}(1-t,1-r) \sigma(1-r) \mathrm{d}r,
	\end{equation*}
	for all $\sigma \in L^1(\hat{I})$.
	Consequently, we arrive at equality \eqref{simetria}.
\end{proof}

Thus, it suffices to analyze the sign of the Green's function $H_{m,M}$ for $M>0$ and $m \neq -M$.

On the other hand, by taking the limit as $m$ tends to zero in the previous expression, we obtain that
\begin{equation*}
	H_{0,M}(t,s) = 
	\begin{cases}
		-\frac{1}{2}+\frac{1}{M}+s-t, & 0 < t < s \leq 1, \\
		\frac{1}{2}+\frac{1}{M}+s-t, & 0 \leq s \leq t \leq 1,
	\end{cases}
\end{equation*}
from which we deduce that Problem \eqref{firstorder} has a unique Green's function $H_{m,M}$, if and only if, $m+M \neq 0$.

Taking this into account, along with the properties of Green's functions and equations \eqref{defg} and \eqref{hfirstsim}, the following lemma is deduced.
\begin{lemma}
	\label{derivada}
	Assume that $m \neq 0$ and $m+M \neq 0$ and let $H_{m,M}$ be the Green's function of Problem \eqref{firstorder}. Then, 
	\begin{equation*}
		\frac{\partial}{\partial{t}}{H}_{m,M}(t,s)=-m \,H_{m,M}(t,s)-\frac{M}{m+M}=-m\,G_{m}(t,s) < 0 \textup{ for all } (t,s) \in \hat{I} \times \hat{I}, \, t \neq s,
	\end{equation*}
	and
	\begin{equation*}
		\frac{\partial}{\partial{s}}{H}_{m,M}(t,s)=m \,H_{m,M}(t,s)+\frac{M}{m+M}=m\,G_{m}(t,s)>0 \textup{ for all } (t,s) \in \hat{I} \times \hat{I}, \, t \neq s.
	\end{equation*}
\end{lemma}
Based on all the above, we present the following lemma, which identifies the points $(t,s) \in \hat{I} \times \hat{I}$ at which the function $H_{m,M}$ attains its minimum when it is positive.
\begin{lemma}
	Assume that $m \neq 0$ and $m+M \neq 0$ and let $H_{m,M}$ be the Green's function of Problem \eqref{firstorder}. Then, for every $s \in (0,1]$, $\inf_{t \in \hat{I}}H_{m,M}(t,s)=H_{m,M}(s^-,s)$.
	\label{lemapos}
\end{lemma}
\begin{proof}
	By applying Lemma \ref{derivada} and using some of the properties that follow from the fact that $H_{m,M}$ is a Green's function, we have that
	\begin{enumerate}
		\item $\frac{\partial}{\partial{t}}H_{m,M} \leq 0$ for all $(t,s) \in \hat{I} \times \hat{I}$, $t \neq s$.
		\item $H_{m,M}(0,s)=H_{m,M}(1,s)$ for all $s \in \hat{I}$.
		\item $\lim_{t \rightarrow s^+} H_{m,M}(t,s) \equiv H_{m,M}(s^+,s)=1+\lim_{t \rightarrow s^-} H_{m,M}(t,s)=1+H_{m,M}(s^-,s)$.
	\end{enumerate}
	It follows that when $H_{m,M}$ is positive, its minimum is attained at a point of the form $(s^-,s) \in \hat{I} \times \hat{I}$.
	
	Finally, noting that using Lemma \ref{derivada} \[\frac{d}{d{s}}H_{m,M}(s^-,s)=\frac{\partial}{\partial{t}}H_{m,M}(s^-,s)+\frac{\partial}{\partial{s}}H_{m,M}(s^-,s)=0,\] 
	the proof is concluded.
\end{proof}

Therefore, using the previous lemma, we can set $H_{m,M}(s^-,s)$ equal to zero for any $s \in (0,1]$, and solve for $M$ as a function of $m$. This leads us to the following theorem:
\begin{theorem}
\label{lema66}
	Assume that $m+M \neq 0$ and let $H_{m,M}$ be the Green's function of Problem \eqref{firstorder}. Then, the following properties are fulfilled:
	\begin{enumerate}
		\item $H_{m,M}(t,s)>0$ for all $(t,s) \in \hat{I} \times \hat{I}$, if and only if, $M \in \left(-m, \frac{m^{2}}{-1+e^{m}-m} \right)$.
		\item If $M=\frac{m^{2}}{-1+e^{m}-m}$ then $H_{m,M}(s^-,s)=0$ for all $s \in (0,1]$.
		\item $H_{m,M}(t,s)<0$ for all $(t,s) \in \hat{I} \times \hat{I}$, if and only if, $M \in \left(-\frac{e^{m}m^{2}}{1-e^{m}+m\,e^{m}},-m \right)$.
		\item If $M=-\frac{e^{m}m^{2}}{1-e^{m}+m\,e^{m}}$ then $H_{m,M}(s^+,s)=0$ for all $s \in [0,1)$.
	\end{enumerate}
	For $m=0$, the endpoints are understood as their limits $2$ and $-2$.
\end{theorem}
Properties $3$ and $4$ can be derived from the symmetry of $H_{m,M}$ given in Lemma \ref{sim2}.

\begin{figure}[H]
	\centering
	\includegraphics[width=0.34\textwidth]{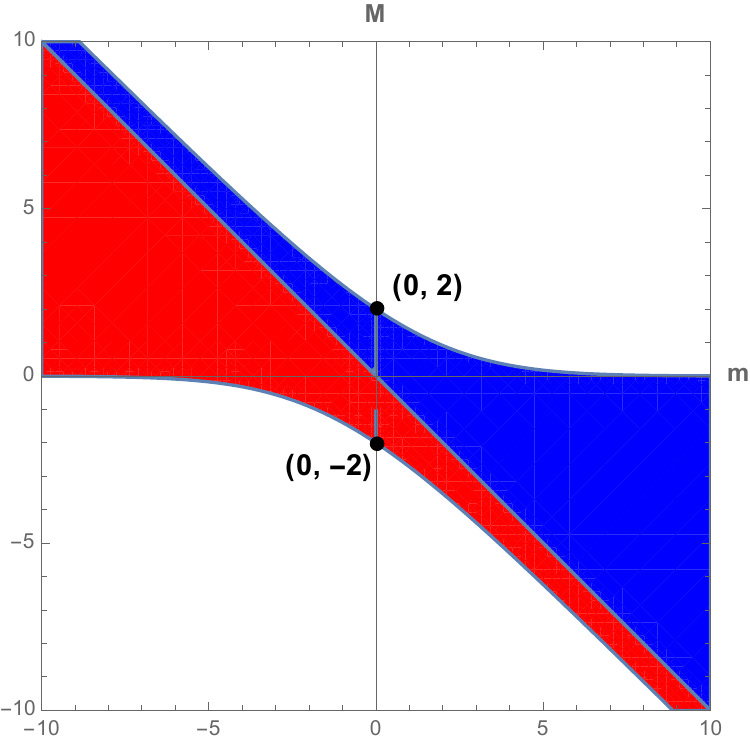}
	\caption{The regions in which the Green's function $H_{m,M}$ for Problem \eqref{firstorder} is positive or negative are depicted in blue and red, respectively.}
	\label{region}
\end{figure}
Theorem \ref{lema66} yields the following result.

\begin{corollary}
	Assume that $m+M \neq 0$ and let $H_{m,M}$ be the Green's function of Problem \eqref{firstorder}. Then, the following properties are fulfilled:
	\begin{enumerate}
		\item $H_{m,M}(t,s)>0$ for all $(t,s) \in \hat{I} \times \hat{I}$, if and only if, $M \in (-m, \lambda_1)$ where $\lambda_1$ is the unique eigenvalue of the terminal value problem:
		\begin{equation}
			\left\{
			\begin{aligned}
				v'(t)+mv(t)+\lambda_1 \int_0^1{v(s) \mathrm{d}s}&= 0, \quad t \in \hat{I}, \\
				v(1)&=0.
			\end{aligned}
			\right.
			\label{finalfacil}
		\end{equation}
		\item $H_{m,M}(t,s)<0$ for all $(t,s) \in \hat{I} \times \hat{I}$, if and only if, $M \in (\lambda_2,-m)$ where $\lambda_2$ is the unique eigenvalue of the initial value problem:
		\begin{equation}
			\left\{
			\begin{aligned}
				v'(t)+mv(t)+\lambda_2 \int_0^1{v(s) \mathrm{d}s}&=0, \quad t \in \hat{I}, \\
				v(0)&=0.
			\end{aligned}
			\right.
			\label{inicialfacil}
		\end{equation}
	\end{enumerate}
\end{corollary}
\begin{proof}
	Part $1$ can be verified directly as a consequence of Part $2$ of Theorem \ref{lema66}. The Green's function $H_{m,M}(t,1)$ satisfies the differential equation by the inherent properties of Green's function. Furthermore, it satisfies $H(1^-,1)=0$, thereby fulfilling the conditions of Problem \eqref{finalfacil}. Meanwhile, Part $2$ is derived analogously from Part $4$ of the aforementioned Theorem.

A direct computation, solving $v'+m\,v=-\lambda \int_0^1 v(s) \mathrm{d}s$
 with $v(1)=0$ (resp. $v(0)=0$), gives
\begin{equation*}
	\lambda_1=\frac{m^2}{-1+e^m-m}, \qquad \qquad \lambda_2=-\frac{e^mm^2}{1-e^m+me^m}.
\end{equation*}
\end{proof}
\begin{remark}
\label{remnolin}
	Finally, we note that the explicit knowledge of the Green's function, together with
	information about its sign, provides a useful tool for studying the
	corresponding nonlinear problem. More precisely, a nonlinear problem
	of the form
	\[
	K_n u(t)=f(t,u(t)),
	\]
	subject to the same boundary conditions, can be rewritten as the
	Hammerstein integral equation
	\[
	u(t)=\int_a^b H(t,s)f(s,u(s))\,\mathrm{d}s.
	\]
	Therefore, under suitable assumptions on the nonlinear term, fixed
	point theory can be used to establish the existence of solutions.
	
	In particular, if the Green's function has a constant sign, the
	associated Hammerstein operator may leave a suitable cone invariant,
	thus allowing fixed point theorems in cones and fixed point index
	techniques to be applied to the study of the existence of positive
	solutions and, under additional assumptions, their multiplicity.
	Furthermore, suitable bounds for the Green's function may provide
	a priori estimates and sufficient conditions for uniqueness.
	Consequently, the analysis carried out above provides the linear
	framework required for the study of the associated nonlinear problem.
	In particular, this approach may be used to develop results analogous
	to those obtained in \cite{cabada2026firstorder}, suitably adapted to the
	nonlinear problem considered here.
\end{remark}
\section{Conclusions}
We have obtained the Green's functions for linear differential equations perturbed by linear functionals and described their adjoint problems. This provides a link between equations with piecewise constant arguments, impulsive equations with integral impulses, and problems subject to nonlocal boundary conditions. Since $H^*(t,s)=H(s,t)$, explicit formulas and constant-sign regions for the Green's functions can be carried over from one setting to another. Moreover, together with the results in \cite{Lab}, this relation gives the exact constant-sign region for some higher order impulsive problems with integral impulses.

These results also suggest several directions for further study. Among them are the nonlinear problems mentioned in Remark \ref{remnolin}, which could be approached using fixed point theorems in cones and the study of operators with deviating arguments.

\appendix
\appendix
\section{Explicit Green's Functions of the Examples}
\label{apendice}

In this appendix, we give the explicit expressions of the Green's functions of Problems \eqref{simedio}, \eqref{ejint} and \eqref{fred}, which are depicted in Figures \ref{e1}, \ref{e2} and \ref{e3}, respectively. In the three cases, the unperturbed problem is $u'(t)+u(t)=\sigma(t)$, $t\in[-2,2]$, $u(-2)=0$, whose Green's function is
\begin{equation}
	\label{gapendice}
	G(t,s)=\left\{
	\begin{aligned}
		e^{s-t}, \quad &\textup{if } -2\leq s\leq t\leq 2,\\
		0, \quad &\textup{if } -2\leq t<s\leq 2,
	\end{aligned}
	\right.
\end{equation}
and, by Theorem \ref{teoprincipal}, the Green's function of each problem can be written as
\begin{equation}
	\label{hapendice}
	H(t,s)=G(t,s)-R(t,s),\qquad (t,s)\in[-2,2]\times[-2,2],
\end{equation}
where $R$ is the term due to the functional part, which we give explicitly below.

\subsection{Problem with a Point Evaluation}
\label{apsimedio}
For Problem \eqref{simedio}, by Corollary \ref{casosimple}, we have
\begin{equation*}
	R(t,s)=\left\{
	\begin{aligned}
		\frac{e^{s+2}-e^{s-t}}{2e^{5/2}-1}, \quad &\textup{if } -2\leq s\leq \tfrac{1}{2},\\
		0, \quad &\textup{if } \tfrac{1}{2}<s\leq 2.
	\end{aligned}
	\right.
\end{equation*}

\subsection{Problem with Integral Dependence on the Partition Intervals}
\label{apsimedio2}
For Problem \eqref{ejint}, we have
\begin{equation*}
	R(t,s)=\left\{
	\begin{aligned}
		&\frac{(e^{-1}-e^{s})(e^{2}-e^{-t})}{e+1}, && \textup{if } s,t\in[-2,-1],\\
		&\frac{(e-e^{s})(e-e^{-t})}{1+2e^{2}}, && \textup{if } s,t\in[-1,1],\\
		&\frac{(e^{2}-e^{s})(e^{-1}-e^{-t})}{e+1}, && \textup{if } s,t\in[1,2],\\
		&\theta\Bigl[2e^{2}(e+1)e^{s}+3e^{1-t}-(4e^{2}+2e+1)e^{s-t}-e^{2}+1\Bigr], && \textup{if } s\in[-2,-1],\ t\in[-1,1],\\
		&\theta\Bigl[3e^{1+s}+2e^{2}(e+1)e^{-t}-(4e^{2}+2e+1)e^{s-t}-e^{2}+1\Bigr], && \textup{if } s\in[-1,1],\ t\in[1,2],\\
		&\frac{\theta}{e+1}\Bigl[(e-1)(e^{2}-4)-2e^{2}(e^{2}-4)\bigl(e^{s}+e^{-t}\bigr)\\
		&\qquad\quad+\bigl(4e^{5}+4e^{4}-10e^{3}-6e^{2}-3e-1\bigr)e^{s-t}\Bigr], && \textup{if } s\in[-2,-1],\ t\in[1,2],\\
		&0, && \textup{otherwise},
	\end{aligned}
	\right.
\end{equation*}
where $\theta=\dfrac{e-1}{(e+1)(1+2e^{2})}$. It is not difficult to verify that $R$ is continuous on $[-2,2]\times[-2,2]$ and that $H(t,s)=H(-s,-t)$ for all $(t,s)\in[-2,2]\times[-2,2]$.

\subsection{Fredholm Integro-Differential Problem}
\label{apsimedio3}
For Problem \eqref{fred}, the full expression of $R$ is too long, so we write it in terms of some auxiliary functions:
\begin{equation*}
	R(t,s)=\frac{1}{\det(A+I_2)}\Bigl[(1+a_{22})\,w_1(t)\,\kappa_1(s)-a_{12}\,w_1(t)\,\kappa_2(s)-a_{21}\,w_2(t)\,\kappa_1(s)+(1+a_{11})\,w_2(t)\,\kappa_2(s)\Bigr],
\end{equation*}
where
\begin{equation*}
	\begin{aligned}
		w_1(t)&=\tfrac12\bigl(\sin t-\cos t+e^{-t-2}(\sin 2+\cos 2)\bigr), &
		w_2(t)&=\tfrac35\bigl(t^2-2t+2-10\,e^{-t-2}\bigr),\\
		\kappa_1(s)&=\tfrac12\bigl(\cos s-\sin s+e^{s-2}(\sin 2-\cos 2)\bigr), &
		\kappa_2(s)&=(2-s)\,e^{s},
	\end{aligned}
\end{equation*}
and $a_{ij}$ are the entries of the matrix $A$ given by \eqref{defa}, that is,
\begin{equation*}
	\begin{aligned}
		a_{11}&=-\tfrac34-\tfrac{\cos 4}{4e^{4}}, &
		a_{12}&=\tfrac35\bigl(3(\sin 2+\cos 2)-5e^{-4}(\sin 2-\cos 2)\bigr),\\
		a_{21}&=-\tfrac{e^{2}\cos 2}{2}+e^{-2}\bigl(2\sin 2+\tfrac52\cos 2\bigr), &
		a_{22}&=\tfrac35\bigl(2e^{2}-58e^{-2}\bigr),
	\end{aligned}
\end{equation*}
so that $\det(A+I_2)=(1+a_{11})(1+a_{22})-a_{12}a_{21}\approx-0.0339$. Notice that $w_i(t)=\int_{-2}^{2}G(t,r)\,\phi_i(r)\,\mathrm{d}r$ and $\kappa_i(s)=\int_{-2}^{2}\xi_i(t)\,G(t,s)\,\mathrm{d}t$, $i=1,2$.

\section*{Declaration of generative AI and AI-assisted technologies\\ in the manuscript preparation process}
During the preparation of this work the authors used ChatGPT and Gemini in order to improve the English of the paper, and Claude (Anthropic) in order to check computations, to suggest improvements to the wording and to help reorganize the presentation of Section 5. The mathematical ideas and results are the authors' own. After using these tools, the authors reviewed and edited the content as needed and take full responsibility for the content of the published article.

	\section*{Declarations}
\noindent\textbf{Competing interests.} The authors declare that they have no competing interests.

\noindent\textbf{Data availability.} Data sharing is not applicable to this article as no datasets were generated or analysed during the current study.

\section*{Acknowledgements}
Project PID2025-169849NB-I00 funded by MICIU/AEI/\allowbreak 10.13039/501100011033, and Project ED431C 2023/12 funded by Xunta de Galicia (Spain). 

P.C.-F. would like to express her gratitude to the Spanish Ministry of Science, Innovation and Universities for financial support (Grant reference FPU 23/02202).

\bibliography{bibno2}
\bibliographystyle{spmpsciper}

\end{document}